%% file: GGP-IMC.tex
\documentclass{amsart}

\input{Preamble}

\def\bs{\boldsymbol}

\newcommand{\Sh}{\mathrm{Sh}}

\newcommand{\CM}{\mathcal{K}}
\newcommand{\con}{\mathtt{c}}

\usepackage{hyperref}

\begin{document}

\title{Wall crossing in Iwasawa theory}
\author{Shilin Lai} 
\address{
Department of Mathematics, the University of Texas at Austin, 2515 Speedway, PMA 8.100, Austin, TX 78712, USA.} 
\email{shilin.lai@math.utexas.edu}

\begin{abstract}
	This paper sets up a framework to organize anticyclotomic Iwasawa theory in the context of the Gan--Gross--Prasad conjecture for unitary groups. We propose multiple main conjectures depending on archimedean weight interlacing conditions, generalizing phenomena in the anticyclotomic Iwasawa theory of elliptic curves. We also prove an abstract theorem in Galois cohomology which relates the conjectures.
\end{abstract}

\maketitle
\setcounter{tocdepth}{1}
\tableofcontents

\section{Introduction}
In studying the arithmetic of elliptic curves, Heegner points play many important role. In the ordinary case, the Heegner points fit into a $p$-adic family over the anticyclotomic $\Z_p$-extension of $\CM$. This family has been used to deduce two kinds of results:
\begin{itemize}
  \item rank 1: If the Heegner point is non-torsion, then the seminal work of Kolyvagin \cite{Kolyvagin88} proves that the Selmer group has rank 1 using an Euler system argument. This is extended to prove one divisibility of Perrin-Riou's Heegner point main conjecture \cite{PRHeegnerPoint} by Howard \cite{HowardDivisibility}, leading to results on the Birch--Swinnerton-Dyer formula in rank 1.
  \item rank 0: By $p$-adically deforming the family and using a special value formula proven by Bertolini--Darmon--Prasanna \cite{BDP}, Castella--Hsieh proved the vanishing of a Selmer group in the region where the specialization is non-geometric \cite{CH18}. They also proved one divisibility of the corresponding Greenberg--Iwasawa main conjecture.
\end{itemize}
In other words, a special point controls two different flavours of Iwasawa theory. From an automorphic point of view, the setting is conjugate self-dual, and the two different regions arise because of a change in the global root number from $-1$ to $+1$. Their interplay continues to play an important role in the arithmetic of elliptic curves, for example in the works of Skinner and Jetchev--Skinner--Wan towards rank 1 cases of the BSD conjecture \cite{SkinnerConverse,JSW}. 

The goal of this article is to propose a framework which generalizes this phenomenon to a higher dimensional setting, motivated by automorphic considerations arising from the Gan--Gross--Prasad conjecture and its arithmetic counterpart \cite{GGPConjecture}.

\subsection{Main results}
Let $\CM/\sh{F}$ be a CM extension, and let $\Pi$ be an regular algebraic, conjugate self-dual, cuspidal (RACSDC) automorphic representation on $\GL_{n-1}(\A_\CM)\times\GL_n(\A_\CM)$. Let $p$ be a prime which splits completely in $\CM$. Suppose $\Pi$ is ordinary at $p$. We are interested in the Iwasawa theory as $\Pi$ varies in the full Hida family of $2n-1$ variables. Let $\bs{T}$ be the associated $p$-adic family of Galois representations. It is a projective module over the ordinary Hecke algebra $\mathbb{I}$ attached to the Hida family. Its properties are axiomatized in \S\ref{ss:BigGaloisRep}.

At each archimedean place of $\sh{F}$, the weight of $\Pi$ there can be labelled by two tuples of half-integers
\[
  (a_1>\cdots>a_n)\in\Big(\Z+\frac{n-1}{2}\Big)^n,\quad (b_1>\cdots>b_{n-1})\in\Big(\Z+\frac{n}{2}\Big)^{n-1}
\]
The relative position of these weights plays an important role. We describe them using a weight interlacing string $\square$ (Definition~\ref{def:String}). Each $\square$ gives rise to a Selmer condition at primes above $p$. When $n=2$ and $\sh{F}=\Q$, the three possible interlacing conditions are
\[
  AAB: a_1>a_2>b_1,\quad ABA: a_1>b_1>a_2,\quad BAA: b_1>a_1>a_2
\]
In the middle case $ABA$, one example for $\Pi$ is just an elliptic curve which is ordinary at $p$, and the Selmer condition is the ordinary one \cite[\S 3.2.1]{SkinnerAWS}. If we twist the elliptic curve by an infinite order anticyclotomic Hecke character, then we are in one of the other two cases. The corresponding Selmer condition is strict at one place above $p$, and relaxed at the other place, leading to the BDP-Selmer group defined in \cite[\S 3.2.3]{SkinnerAWS}.

Now consider the $p$-adic family $\bs{T}$. For each string $\square$, its Selmer condition interpolates to give one for $\bs{T}$, and we can define a Greenberg-style Selmer group (more precisely the cohomology group of a Selmer complex, cf.~\S\ref{ss:IwAlg})
\[
  \h^2_\square(\CM,\bs{T}).
\]
When specialized to a classical point $x$ whose weight interlacing string is actually $\square$, this recovers the Bloch--Kato Selmer group. However, in a Hida family, the weight interlacing relation is not fixed, so we should consider all such Selmer groups, one for each choice of $\square$. As $\square$ varies, the only change in the Selmer condition is at the primes above $p$.

Each $\square$ should lead to its own Iwasawa main conjecture, but the nature of the conjecture differs depending on $\square$. In fact, we expect that the rank of $\h^2_\square(\CM,\bs{T})$ could be either 0 or 1 depending on $\square$. To explain this, note that the global root number of a specialization depends only on its weight interlacing string (Lemma~\ref{lem:Constant}). As a result, for ``half'' of the strings $\square$, the global root number is $-1$, and the Bloch--Kato conjecture predicts that the Selmer groups for those specializations have rank at least 1. In this case, $\h^2_\square(\CM,\bs{T})$ is expected to have rank 1 over $\mathbb{I}$. This is in contrast to non-self-dual settings such as cyclotomic deformation \cite{GreenbergMotives}, where the rank is expected to be 0.

To formalize the above discussions, we propose the following framework.
\begin{conj}[Conjecture~\ref{conj:MainConjecture}]
  Suppose $p$ splits completely in $\CM$ and $\Pi$ is ordinary at $p$. Given a $p$-adic family of Galois representations $\bs{T}$ as above, there are $\binom{2n-1}{n}^{[\sh{F}:\Q]}$ main conjectures indexed by weight interlacing strings $\square$. They are divided into two types depending only on a sign attached to $\square$:
  \begin{itemize}
    \item \textbf{Coherent case} The Selmer group $\h^2_\square(\CM,\bs{T})$ is torsion over $\mathbb{I}$, and
    \[
      \Char_{\mathbb{I}}\h^2_\square(\CM,\bs{T})=(\sh{L}_p^\square)^2,
    \]
    where $\sh{L}_p^\square$ is a $p$-adic $L$-function whose square interpolates central critical values of classical points when their archimedean weights lie in the region defined by $\square$.
    \item \textbf{Incoherent case} The Selmer group $\h^2_\square(\CM,\bs{T})$ has rank 1 over $\mathbb{I}$, and there is a special class
    \[
      \bs{z}_\square\in\h^1_\square(\CM,\bs{T})
    \]
    such that
    \[
      \Char_{\mathbb{I}}\h^2_\square(\CM,\bs{T})_\mathrm{tors}=\Char_{\mathbb{I}}\left(\frac{\h^1_\square(\CM,\bs{T})}{\mathbb{I}\cdot\bs{z}_\square}\right)^2
    \]
  \end{itemize}
\end{conj}

\begin{remark}
  In the related orthogonal GGP setting $\mathrm{SO}_4\times\mathrm{SO}_5$, analogous conjectures have been formulated by Loeffler--Zerbes \cite{LZConjecture}. 
\end{remark}

The special class should play an important role for both the incoherent and coherent main conjecture. We explain this as a wall crossing phenomenon. In the $n=2$ case, if we change the weight interlacing relation from $(a_1>a_2>b_1)$ to $(a_1>b_1>a_2)$, then the local root number at the archimedean place changes, so we transition between a coherent and an incoherent region. This is an example of a pair of \emph{nearby} interlacing relations. In general, there is a relation between the main conjectures for nearby words.

\begin{theorem}[Theorem~\ref{thm:Equiv}]
  Assume $\mathbb{I}$ is Gorenstein. Suppose $\square$ and $\triangle$ are nearby (Definition~\ref{def:Nearby}), and $\triangle$ is incoherent, then there is a regulator map
  \[
    \mathrm{reg}:\widetilde{\h}^1_\triangle(\CM,\bs{T})\to\mathbb{I}
  \]
  If $\mathrm{reg}(\bs{z}_\triangle)\neq 0$ for some $\bs{z}_\triangle\in\widetilde{\h}^1_\triangle(\CM,\bs{T})$, then the incoherent main conjecture for $\triangle$ and special cycle $\bs{z}_\triangle$ is equivalent to the coherent main conjecture for $\square$, with the $p$-adic $L$-function $\sh{L}_p^\square$ replaced by the motivic $p$-adic $L$-function defined by $\sh{L}_\mathrm{mot}^\square:=\mathrm{reg}(\bs{z}_\triangle)$.
\end{theorem}

The main input to this theorem is the construction of the regulator map, which comes down to the one-dimensional Coleman map. The proof of the equivalence then involves a series of applications of global duality theorems, which we organize using Selmer complexes.

\subsection{Related and future works}
We know describe some cases of the conjecture which have been studied in the literature. Along the way, we indicate a few questions for future investigations.
\subsubsection{Diagonal condition}
We first consider the ``diagonal'' interlacing condition
\[
  \mathtt{diag}:a_1>b_1>a_2>\cdots>b_{n-1}>a_n.
\]
For a single representation $\Pi$ with this interlacing condition, the paper \cite{LTXZZ} made significant progresses towards rank 0 and 1 cases of the Bloch--Kato conjecture. For the partial family of $\Pi$ twisted by Hecke characters, the $p$-adic $L$-function $\sh{L}_p^\mathtt{diag}$ in the coherent case and the special class $\bs{z}_\mathtt{diag}$ in the incoherent case were both constructed by Yifeng Liu \cite{LiuRS}. Under some technical hypotheses, the divisibility (LHS)$|$(RHS) of both the coherent and the incoherent main conjecture was recently established by Liu--Tian--Xiao \cite{LTX}.

The theory of split anticyclotomic Euler systems by Jetchev--Nekov\'{a}\v{r}--Skinner \cite{JNS} offers an alternative approach to establishing cases of the incoherent main conjecture. The necessary Euler system input was recently constructed by the author and Skinner \cite{LaiSkinner}, giving a different proof the incoherent divisibility of Liu--Tian--Xiao. It is likely one make this approach work over the full Hida family using the method described in \cite{LRZ-Moment}. We plan to investigate this in a future work.

When $n=3$ and $\sh{F}=\Q$, the full 5-variable $p$-adic $L$-function $\sh{L}_p^\mathtt{diag}$ has been constructed by Hsieh--Yamana \cite{HsiehYamana23}. It would be interesting to study if the method of \cite{LTX} can be extended to prove one divisibility of the 5-variable coherent main conjecture.

\subsubsection{Other conditions}
Starting with $\mathtt{diag}$, there are $2(n-1)$ nearby words, one example being
\[
  \square:b_1>a_1>a_2>b_2>\cdots>b_{n-1}>a_n,
\]
which we obtained from $\mathtt{diag}$ by switching the order of $a_1$ and $b_1$. In the case where $\mathtt{diag}$ is incoherent, all such nearby words are coherent. Once the full Hida family deformation $\bs{z}_\mathtt{diag}$ is constructed, our theorem relates the incoherent main conjecture for $\mathtt{diag}$ to coherent main conjectures for all $2(n-1)$ nearby words.

To exploit this relation, we need to relate the regulator of $\bs{z}_\mathtt{diag}$ to a $p$-adic $L$-function obtained by interpolation, or in other words, prove an ``explicit reciprocity law''. In the case of $\square$ given above, a 1-variable $p$-adic $L$-function has been constructed by Harris \cite{HarrisSqRoot}. As explained in Example~\ref{ex:DiagNearby}, this is exactly the variable needed to deform the word $\square$ into $\mathtt{diag}$. The explicit reciprocity law in this case should be within reach of current methods.

On the other hand, when the word $\mathtt{diag}$ is coherent, its $p$-adic $L$-function should be related to special classes for nearby words such as $\square$. However, currently there is no construction of these special classes except for $n=2$, which is a recent work of Castella--Do \cite{CastellaTuan}. One way to replicate their construction in our setting is to allow non-cuspidal $\Pi$. In the product $\Pi=\Pi_{n-1}\times\Pi_n$, we can take $\Pi_n$ to be an isobaric sum of the form $\Pi_{n-2}\boxplus\chi_1\boxplus\chi_2$ for two auxiliary characters $\chi_1,\chi_2$. Then the Galois representation for $\Pi$ decomposes as a direct sum, where one piece is associated to the lower rank case $\Pi_{n-2}\times\Pi_{n-1}$. It is a subtle and interesting question to understand this process in family.

Finally, some computations in \S\ref{sec:Analytic} raises a possibility that there may be further connections between the regions of interpolation. This already appears in the $n=3$ case, and we hope this can be clarified in the near future.

\subsection{Notations and conventions}\label{sec:Notation}
Let $\sh{F}$ be a totally real field of degree $g$, and let $\CM$ be a CM extension of $\sh{F}$. Let $\mathtt{c}$ be a fixed element in $\Gal_{\sh{F}}$ which is non-trivial in $\Gal(\CM/\sh{F})$. Fix a CM type $\Sigma^+_\infty$ for $\CM$. We will typically use it to index the archimedean places of $\sh{F}$, so for each archimedean place, we have a distinguished embedding $\CM\hookrightarrow\C$.

Let $p>2$ be a rational prime such that
\begin{equation}\tag{spl}
  p\text{ splits completely in }\CM
\end{equation}
Fix an isomorphism $\iota:\C\simeq\C_p$. Let $\Sigma_p^+=\{\iota\circ\sigma\,|\,\sigma\in\Sigma^+_\infty\}$. This is a $p$-adic CM type, and we may identify $\Sigma_p^+$ with a set of $p$-adic places of $\CM$. Moreover, $\Sigma_p^+\cap\Sigma_p^+\mathtt{c}=\emptyset$, so for each $p$-adic place $v$ of $\sh{F}$, the CM type distinguishes one of the two places above $v$.

Let $\A$ denote the ring of adeles of $\sh{F}$ and $\A_\sh{K}$ denote the ring of adeles of $\sh{K}$. We will systematically use the following convention: a superscript consisting of places indicates omitting those places, and a subscript consisting of places indicates only considering those places. For example, $\A_{\sh{K},f}$ is the ring of finite adeles of $\sh{K}$, and $\A_f^p=\A^{p\infty}$ is the ring of finite adeles of $\sh{F}$ omitting the places above $p$.

Class field theory is always normalized geometrically. We will implicitly use it to pass between characters of the Galois group and algebraic Hecke characters. Hodge--Tate weights are normalized so that the cyclotomic character has weight $-1$.

\subsection*{Acknowledgement}
This article is a part of the author's PhD thesis, and I would like to thank my advisor Christopher Skinner for suggesting this problem and advising me throughout the process. I would also like to thank Ashay Burungale, Francesc Castella, Zheng Liu, and Wei Zhang for their interests, helpful discussions, and constant encouragement.

\section{\texorpdfstring{$p$}{p}-adic family of Galois representations}\label{sec:1}
In this section, we set up some notations and definitions for Iwasawa theory. In particular, we construct a Coleman map in our setting. Fix $E$, a finite extension of $\Q_p$. Let $\Or$ be its ring of integers. They will serve as coefficients for the representations we consider.

\subsection{Weight space}\label{ss:WeightSpace}
Let $\mathbf{T}=\Res_{\sh{F}/\Q}\G_m^r$. Let $\underline{\Lambda}=\Or[[\mathbf{T}(\Z_p)]]$ be its completed group ring. This is a semi-local regular Noetherian ring. Let $\mb{T}^0$ be the pro-p part of the abelian group $\mb{T}(\Z_p)$, then we have a factorization
\[
  \mb{T}(\Z_p)=\mb{T}^0\times\Delta
\]
for some finite abelian group $\Delta$. Correspondingly, we can write $\underline{\Lambda}=\Lambda[\Delta]$, where $\Lambda$ is the completed group ring of $\mb{T}^0$. Given a character $\psi:\Delta\to\Or^\times$, we can pair with it to get a surjection $\mathtt{pr}_\psi:\underline{\Lambda}\to\Lambda$. This presents $\Lambda$ as a $\underline{\Lambda}$-algebra, and we will write $\Lambda^\psi$ if we want to emphasize this structure. 

Let $\underline{\sh{W}}$ be the rigid analytic space attached to $\spf\underline{\Lambda}$, then it is a disjoint union of open unit discs labelled by characters of $\Delta$. Let $\sh{W}_\psi$ denote the component corresponding to $\psi$. We will view $\underline{\Lambda}$ as the bounded functions in the space of rigid functions $\Or(\underline{\sh{W}})$.

Given a cocharacter $\epsilon:\Res_{\sh{F}/\Q}\G_m\to\mathbf{T}$, we can define the tautological character
\begin{equation}\label{eqn:TautChar}
  \bs{\chi}_\epsilon:\Gal_\CM\twoheadrightarrow\Gal(\CM^\mathrm{ac}/\CM)\to\Res_{\sh{F}_p/\Q_p}\Z_p^\times\to\underline{\Lambda}^\times,
\end{equation}
where the final step sends $z$ to the group element $[\epsilon(z)]$. We say a specialization $x:\underline{\Lambda}\to\bar{\Q}_p$ is \emph{classical} if $x\circ\bs{\chi}_\epsilon$ is Hodge--Tate for all $\epsilon$. In this case, we define the $\epsilon$-weight $w_\epsilon(x)\in\Z^g$ to be negative of the Hodge--Tate weight of $x\circ\bs{\chi}_\epsilon$ at the places of $\CM$ lying in the CM type $\Sigma_p^+$. In particular, by taking $\varepsilon$ to be the $r$ standard cocharacters, we get a tuple of weights
\[
  (w_1^\sigma,\cdots,w_r^\sigma)_{\sigma\in\Sigma_p^+}\in\prod_{\sigma\in\Sigma_p^+}\Z^r
\]
attached to $x$. The subset of classical points of $\underline{\sh{W}}$ will be denoted by $\underline{\sh{W}}^\mathrm{cl}$. 

\subsection{\texorpdfstring{$p$}{p}-adic family} Let $\underline{\mathbb{I}}$ be a ring with a finite inclusion $\underline{\Lambda}\to\underline{\mathbb{I}}$. Let $\underline{\sh{E}}$ denote the rigid space attached to $\spf\underline{\mathbb{I}}$, so we have a finite morphism $\underline{\sh{E}}\to\underline{\sh{W}}$. We say a point of $\underline{\sh{E}}$ is classical if its image in $\underline{\sh{W}}$ is classical. In the intended applications, $\underline{\sh{E}}$ occurs as an irreducible component of the ordinary eigenvariety.

\begin{definition}\label{def:Family}
  A $p$-adic family of Galois representations is a finitely generated $\underline{\mathbb{I}}$-module $\bs{T}$ with an action of $\Gal_\CM$ such that
  \begin{enumerate}
    \item The $\Gal_\CM$-action is admissible in the sense of \cite[Section 3.2]{NekovarSC}.
    \item As a $\underline{\Lambda}$-module, $\bs{T}$ is projective.
  \end{enumerate}
\end{definition}

The first condition is a version of continuity. In the cases we will consider, it holds for a finite rank $\underline{\mathbb{I}}$-module if the action is continuous with respect to the topology induced from each of the maximal ideal of $\underline{\Lambda}$. The notion is also stable under Pontryagin duality and extensions, so we will usually not need to explicitly check it in practice.

If $\bs{\chi}$ is a character valued in $\underline{\mathbb{I}}^\times$, then we write $\underline{\mathbb{I}}(\bs{\chi})$ to mean $\underline{\mathbb{I}}$ with the Galois group acting through $\bs{\chi}$. This is admissible if and only if $\bs{\chi}$ is continuous. In that case it defines a $p$-adic family of characters. In particular, for any $\alpha\in\underline{\mathbb{I}}^\times$, we can form the unramified character
\[
  \bs{\alpha}:\Gal_{\Q_p}\to\underline{\mathbb{I}}^\times
\]
sending a geometric Frobenius to $\alpha$. Moreover, using the inclusion $\underline{\Lambda}\to\underline{\mathbb{I}}$, the universal characters $\bs{\chi}_\epsilon$ defined in equation~\eqref{eqn:TautChar} are also characters on $\underline{\mathbb{I}}$.

\subsection{Commutative algebra}
The notation $(-)^\vee$ will always mean taking the (underived) dual with respect to a dualizing module over the implied base ring. This is usually applied with $\Z_p$ or $\Lambda$, depending on the context. Both of them are Gorenstein, so they are their own dualizing module.
\begin{definition}\label{def:CSD}
  A $p$-adic family $\bs{T}$ is \emph{conjugate self-dual} if there exists a $\Gal_K$-invariant bilinear pairing $\bs{T}\otimes_\Lambda\bs{T}^\con\to\Lambda(1)$ which induces an isomorphism $\bs{T}^\con\simeq\Hom_\Lambda(\bs{T},\Lambda)(1)$.
\end{definition}
 Now suppose $\bs{T}$ is conjugate self-dual. There are two types of duality pairing we want to consider for $\bs{T}$. For the ring $R=\Lambda$, the representations $\bs{T}$ and $\bs{T}^\con$ satisfy condition 5.2.3(B) of \cite{NekovarSC}, so we get a good theory of duality for their Galois cohomology complexes viewed as $\Lambda$-modules. Let $(-)^*$ denote the Pontryagin dual, then we have $\bs{T}^*\simeq\bs{T}^\vee\otimes_\Lambda\Lambda^*$ as $\Lambda$-modules. The pair $\bs{T}$ and $\bs{T}^\con\otimes_\Lambda\Lambda^*$ therefore satisfies condition 5.2.3(A) of \cite{NekovarSC}.

Finally, we define notations surrounding the characteristic ideal. Let $P$ be a height one prime ideal of a complete local ring $R$ (either $\Lambda$ or $\mathbb{I}$ in applications). It is associated to a valuation $v_P$ on $\Lambda$ which can be computed as $v_P(x)=\mathrm{length}_{R_P}(R_P/(x))$. The characteristic ideal of a torsion $R$-module $X$ is defined to be
\[
  \Char_{R}(X):=\{x\in\Lambda\,|\,v_P(x)\geq\mathrm{length}_{R_P}(X_P)\text{ for all $P$ of height 1}\}
\]
It is clear that the characteristic ideal is multiplicative in exact sequences. A torsion module is pseudo-null if its characteristic ideal is the unit ideal. They form a Serre subcategory of the category of all $R$-modules. We use the notation $\mathbf{psMod}_R$ to denote the quotient abelian category. Though we will not use it, we remark that all torsion $\Lambda$-modules are isomorphic to a direct sum of the form $\bigoplus_i\Lambda/(f_i)$ in $\mathbf{psMod}_\Lambda$.

\subsection{Coleman map}
We want to generalize the classical Coleman logarithm $\h^1(\Q_p,\Lambda(1))\simeq\Lambda$ to our setting. Our result will follow from a slight modification of the constructions in \cite{OchiaiColemanMap}. We first recall some notations. Let $\Q_p^\mathrm{ur}$ denote the maximal unramified extension of $\Q_p$ and $\Gamma_\mathrm{ur}=\Gal(\Q_p^\mathrm{ur}/\Q_p)$. There is a canonical isomorphism $\hat{\Z}\to\Gamma_\mathrm{ur}$ sending $1$ to the geometric Frobenius $\sigma$.

Let $\chi_\mathrm{cyc}:\Gal_{\Q_p}\to\Z_p^\times$ be the cyclotomic character. Let $\epsilon:\G_m\to\mathbf{T}$ be a standard cocharacter, so we have a universal character $\bs{\chi}_\epsilon:\Gal_{\Q_p}\to\underline{\Lambda}^\times$ defined by equation~\eqref{eqn:TautChar}. Let $\bs{\alpha}\in\underline{\mathbb{I}}^\times$, which gives an unramified character $\bs{\alpha}$. We are interested in the character
\[
  \bs{\chi}=\bs{\alpha}\bs{\chi}_\epsilon:\Gal_{\Q_p}\to\underline{\mathbb{I}}^\times.
\]

\begin{theorem}\label{thm:Coleman}
  Suppose that the specialization of $\bs{\alpha}$ at the trivial character is not equal to 1. Then for any finitely generated $\underline{\mathbb{I}}$-module $M$ which is projective over $\Lambda$, there is an isomorphism of $\underline{\mathbb{I}}$-modules
  \[
    \mathrm{Col}: \h^1(\Q_p,M(\bs{\chi}))\to (M(\bs{\alpha})\hat{\otimes}_{\Z_p}\hat{\Z}_p^\mathrm{ur})^{\Gamma_\mathrm{ur}}
  \]
  Moreover, the right hand side is non-canonically isomorphic to $M$.
\end{theorem}
\begin{proof}
  We will work over each component of $\underline{\mathbb{I}}$ separately, so fix a tame character $\psi$ and suppress it from the notation. The group $\Gamma_\mathrm{ur}$ has $p$-cohomological dimension 1, so we have the following exact sequence
  \[
    0\to \h^1(\Gamma_\mathrm{ur},M(\bs{\chi})^{\Gal_{\Q_p^\mathrm{ur}}})\to \h^1(\Q_p,M(\bs{\chi}))\to \h^1(\Q_p^\mathrm{ur},M(\bs{\chi}))^{\Gamma_\mathrm{ur}}\to 0
  \]
  By assumption, $\epsilon$ is non-trivial and $M$ is $p$-torsion-free, so the $\Gal_{\Q_p^\mathrm{ur}}$-invariant is trivial. Moreover, over $\Q_p^\mathrm{ur}$, we can factor out the unramified character $\bs{\alpha}$. Therefore,
  \[
    \h^1(\Q_p,M(\bs{\chi}))\simeq\big(\h^1(\Q_p^\mathrm{ur},\Lambda(\bs{\chi}_\epsilon))\otimes_\Lambda M(\bs{\alpha})\big)^{\Gamma_\mathrm{ur}}
  \]
  Fix an isomorphism $\mathbf{T}\simeq\G_m^r$ such that $\epsilon$ is the first component, then as a Galois module,
  \[
    \Lambda(\bs{\chi}_\epsilon)=\Lambda_1(\chi_\mathrm{cyc})\hat{\otimes}\Lambda_{r-1}
  \]
  where $\Lambda_1\simeq\Z_p[[\Z_p]]$, and the Galois action on the second term is trivial.
  
  We now recall the Coleman map, in the form stated as \cite[Proposition 5.10]{OchiaiColemanMap}. It gives a short exact sequence
  \[
    0\to\Z_p\to \h^1(\Q_p^\mathrm{ur},\underline{\Lambda}_1(\chi_\mathrm{cyc}))\xrightarrow{\,\Log\,}\underline{\Lambda}_1\hat{\otimes}_{\Z_p}\hat{\Z}_p^\mathrm{ur}\to 0
  \]
  such that for any classical point $\eta$, the following diagram commutes
  \[
  \begin{tikzcd}
    \h^1(\Q_p^\mathrm{ur},\underline{\Lambda}_1(\chi_\mathrm{cyc}))\rar["\Log"]\dar["\mathrm{sp}_\eta"] & \underline{\Lambda_1}\hat{\otimes}_{\Z_p}\Z_p^\mathrm{ur}\dar["\mathrm{sp}_\eta"]\\
    \h^1(\Q_p^\mathrm{ur},F_\eta(\eta))\rar["\log_\eta"] & D_\dR^\mathrm{ur}(F_\eta(\eta))
  \end{tikzcd}
  \]
  We need to define $\log_\eta$. Write $\eta=\phi\chi_\mathrm{cyc}^w$, where $\phi$ has finite order and $w\geq 1$ is the weight of $\eta$. Let $p^s$ be the conductor of $\phi$, then
  \[
    \log_\eta(z):=\left(\frac{\sigma^{-1}}{p^{w-1}}\right)^s\left(1-\frac{p^{w-1}\phi(p)}{\sigma^{-1}}\right)\left(1-\frac{\sigma^{-1}\phi(p)}{p^w}\right)\frac{(-1)^{w-1}}{(w-1)!}\log_\mathrm{BK}(z)
  \]
  The $\sigma$ here represents the geometric Frobenius acting on $\Q_p^\mathrm{ur}$, which explains the presence of inverses. 
  
  To complete the proof, it remains to observe that the kernel term only shows up if the tame character $\psi$ is trivial. In this case,
  \[
    (\Z_p\otimes_\Lambda M(\bs{\alpha}))^{\Gamma_{\mathrm{ur}}}=0
  \]
  since $M$ is projective over $\Lambda$, and the reduction of $\bs{\alpha}$ is not trivial. It follows that we have an isomorphism
  \[
    \h^1(\Q_p,M(\bs{\chi}))\simeq(M(\bs{\alpha})\hat{\otimes}_{\Z_p}\hat{\Z}_p^\mathrm{ur})^{\Gamma_\mathrm{ur}}
  \]
  Finally, the proof of \cite[Lemma 3.3]{OchiaiColemanMap} gives an isomorphism of abelian groups between the right hand side and $M$. It is clear from the construction that this preserves the $\mathbb{I}$-module structure.
\end{proof}

\section{Selmer groups and main conjectures}\label{sec:Main}
This is the main technical part of the paper. In a conjugate self-dual situation, we will propose two types of abstract Iwasawa main conjectures. Using a formal computation in Galois cohomology, we will prove their equivalence, assuming the existence of certain ``special elements''. This will be applied in \S\ref{sec:Guess} to bridge between the Iwasawa main conjectures stated there.

\subsection{Ordinary family}{
\newcommand{\an}{\llcorner}
\newcommand{\bn}{\urcorner}
\newcommand{\ad}{{\urcorner}}
\newcommand{\bd}{{\llcorner}}
\renewcommand{\p}{w}

In this section, fix a tame character $\psi$, and hence a component $\Lambda$ of $\underline{\Lambda}$. Let $\bs{T}$ be a conjugate self-dual $p$-adic family of representations. Fix a $p$-adic place $v$ of $\sh{F}$. Let $\p$ be the place of $\CM$ above $v$ which lies in $\Sigma_p^+$. We impose the following condition.
\begin{equation}\tag{$\mathrm{ord}_v$}
  \parbox{\dimexpr\linewidth-5em}
  {There exists two $\Gal_{\CM_\p}$-stable $\Lambda$-projective submodules $\bs{T}^\an\subseteq\bs{T}^{\bn}\subseteq\bs{T}$ such that $\bs{T}^\bn/\bs{T}^\an$ has rank one over $\mathbb{I}$, and its Galois action given by a character of the form $\bs{\alpha}\bs{\chi}_\epsilon$ for some non-trivial cocharacter $\epsilon$.
  }
\end{equation}
Using the isomorphism $\bs{T}^\con\simeq\bs{T}^\vee(1)$, we can dualize the two submodules to obtain a sequence of $\Gal_{\CM_{\bar{w}}}$-stable submodules $\bs{T}^{\con,\ad}\subseteq\bs{T}^{\con,\bd}\subseteq\bs{T}^\con$, so for example, $\bs{T}^{\con,\ad}$ is the transfer of the annihilator of $\bs{T}^{\bn}$ in $\bs{T}^\vee$. We view them as $\Gal_{\CM_{\bar w}}$-stable submodules of $\bs{T}$. The action of $\Gal_{\CM_{\bar{w}}}$ on $\bs{T}^{\con,\bd}/\bs{T}^{\con,\ad}$ is given by $\bs{\alpha}^{-1}\bs{\chi}_{-\epsilon}\chi_\cyc$.

\begin{remark}
  The notations used here is supposed to represent the following picture.
  
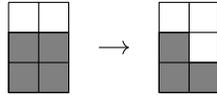
\begin{center}
  \begin{tikzpicture}[scale=0.4]
    \draw [fill=gray] (0,0) rectangle (2,2);
    \draw (0,0) grid (2,3);
    \draw [->] (3,1.5) -- (4,1.5);
    \begin{scope}[xshift=5cm]
	    \draw [fill=gray] (0,0) rectangle (1,2);
	    \draw [fill=gray] (1,0) rectangle (2,1);
	    \draw (0,0) grid (2,3);
    \end{scope}
  \end{tikzpicture}
  \end{center}
  We think of $\bs{T}$ as the whole box, $\bs{T}^\bn$ as the filled in subspace in the first diagram, and $\bs{T}^\an$ as the one for the second diagram, removing a single $\mathbb{I}$-rank-1 piece. The notations on the dual space are chosen so that the Selmer conditions $(\an,\bd)$ and $(\bn,\ad)$ we define later are self-dual.
\end{remark}

The Coleman map gives the following result on the local cohomology of $\bs{T}^\bn/\bs{T}^\an$.
\begin{cor}\label{cor:Coleman}
  There is a non-canonical isomorphism $\RG_\mathrm{cont}(\CM_\p,\bs{T}^\bn/\bs{T}^\an)\simeq(\bs{T}^\bn/\bs{T}^\an)[1]$ in $D^b(\mathbf{psMod}_\Lambda)$.
\end{cor}
\begin{proof}
  Given Theorem~\ref{thm:Coleman}, the claim follows from the statement that $\h^i(\CM_\p,M)$ is pseudo-null if $i=0,2$, where $M$ is an $\mathbb{I}$-module with $\Gal_{\CM_\p}$ acting through a character of the type specified. The case of $i=0$ holds since $\epsilon\neq 1$ and $\bs{T}$ is $\Lambda$-projective. By the same reason, it holds for $M^\vee(1)$. We prove the $i=2$ part for $M$ and $M^\vee(1)$ at the same time.
  
  By local duality in the form of \cite[Theorem 5.2.6]{NekovarSC}, we have an isomorphism $\RG_\cont(\CM_\p,M^\vee(1))\simeq\RHom(\RG_\cont(\CM_\p,M),\Lambda)[2]$. After applying the known values for degrees 0 and 1, the spectral sequence computing the right hand side degenerates to
  \[
    \Ext^i_\Lambda(\h^2(\CM_\p,M),\Lambda)=\begin{cases}\h^2(\CM_\p,M^\vee(1)) & i=2\\0 & \text{otherwise}\end{cases}
  \]
  and similarly with the roles of $M$ and $M^\vee(1)$ exchanged. When localized at a height one prime, the left hand side vanishes for $i\geq 2$ since $\Lambda$ is regular. This proves the required conclusion.
\end{proof}

\subsection{Selmer groups and complexes}\label{sec:Selmer}
In this section, we recall various general facts about Selmer groups and Selmer complexes, in particular setting up notations for later use. Fix a finite set $S$ of finite places of $\CM$ such that $\bs{T}$ is unramified away from $S$ and $\infty$. Suppose also $S$ contains all places above $p$.
\begin{definition}
  Let $\sh{F}$ be a Selmer structure for $\bs{T}$, namely a choice of a subspace $\sh{F}_v\subseteq \h^1(K_v,\bs{T})$ for each $v\in S$. The compact Selmer group is
  \[
    \h^1_{\sh{F}}(\CM,\bs{T}):=\ker\left(\h^1(\CM,\bs{T})\to\bigoplus_v \h^1(\CM_v,\bs{T})/\sh{F}_v\right)
  \]
  Let $\sh{F}_v'$ be the image of $\sh{F}_v$ under the map $\h^1(\CM_v,\bs{T})\to \h^1(\CM_v,\bs{T}\otimes_\Lambda\Lambda^*)$. The Greenberg Selmer group is
  \[
    \operatorname{Sel}_{\sh{F}}(\CM,\bs{T}):=\ker\left(\h^1(\CM,\bs{T}\otimes_\Lambda\Lambda^*)\to\bigoplus_v \h^1(\CM_v,\bs{T}\otimes_\Lambda\Lambda^*)/\sh{F}_v'\right)
  \]
  Its Pontryagin dual will be denoted by $X_{\sh{F}}(\CM,\bs{T})$. 
  
  Fix a choice of Selmer conditions at the places above $p$ which are not $w$ or $\bar{w}$. At $w$, we will use $\spadesuit\in\{\an,\bn\}$ to denote the (strict) Greenberg-type Selmer condition
  \[
    \sh{F}^\p_\spadesuit=\ker(\h^1(\CM_\p,\bs{T})\to \h^1(\CM_\p,\bs{T}/\bs{T}^\spadesuit))
  \]
  The same symbols $\{\ad,\bd\}$ will also be used to denote the dual Selmer conditions at $\bar\p$. Unless otherwise stated, the Selmer condition at a place away from $p$ is the unramified subspace $\h^1(\F_v,\bs{T}^{I_v})$, viewed as a subspace of $\h^1(\CM_v,\bs{T})$ by inflation. Therefore, we will use variations of the notation $\h^1_{\spadesuit,\clubsuit}$ to denote $\spadesuit$-condition at $\p$, $\clubsuit$-condition at $\bar{\p}$, and the unramified condition everywhere else.
\end{definition}

\begin{definition}
  A local condition $\Delta_v$ is a map of complexes $U^+_v\to C_\cont^\bullet(\Gal_{\CM_v},\bs{T})$. Given a collection $\Delta=(\Delta_v)_{v\in S}$, the associated Selmer complex is
  \[
    \widetilde{\RG}_\Delta(\CM,\bs{T}):=\operatorname{Cone}\left(C_\cont^\bullet(\Gal_{\CM,S},\bs{T})\oplus\bigoplus_{v\in S}U^+_v\xrightarrow{\,\res_v-\Delta_v\,}\bigoplus_{v\in S}C_\cont^\bullet(\Gal_{\CM_v},\bs{T})\right)[-1]
  \]
  Its cohomology groups will be denoted by $\widetilde{\h}^\bullet_\Delta(\CM,\bs{T})$.
  
  Similarly as before, for $\spadesuit\in\{\an,\bn\}$, we can define the local condition
  \begin{equation}\label{eqn:GreenbergComplex}
    \Delta_\spadesuit^{\p}:C_\cont^\bullet(\Gal_{\CM_\p},\bs{T}^\spadesuit)\to C_\cont^\bullet(\Gal_{\CM_\p},\bs{T})
  \end{equation}
  at $\p$, and for $\clubsuit\in\{\ad,\bd\}$ at , we can define the condition $\Delta_\clubsuit^{\bar\p}$ at $\bar\p$. Away from $p$, we use the unramified local condition $C_\cont^\bullet(\Gal_{\CM_v}/I_v,\bs{T}^{I_v})\to C_\cont^\bullet(\Gal_{\CM_v},\bs{T})$.
\end{definition}

By construction, the subspaces $\bs{T}^\an$ and $\bs{T}^{c,\bd}$ are mutual annihilators. It follows from \cite[Proposition 6.7.6]{NekovarSC} that $\Delta_\an^\p$ and $\Delta_{\bd}^{\bar{\p}}$ are orthogonal complements at the complex level. The same discussion holds with all $\an$ replaced by $\bn$.

One version of global duality can be stated cleanly in this language. We have a Pontryagin duality pairing $\bs{T}\otimes_\Lambda(\bs{T}^c\otimes_\Lambda\Lambda^*)\to\Lambda^*(1)$. From \cite[\S 7.8.4.3]{NekovarSC}, we get an isomorphism
\begin{equation}\tag{$\ast$}
  \widetilde{\RG}_{\spadesuit,\clubsuit}(\CM,\bs{T})\simeq\Hom_\cont(\widetilde{\RG}_{\clubsuit,\spadesuit}(\CM,\bs{T}\otimes_\Lambda\Lambda^*),\Q_p/\Z_p)[-3]
\end{equation}
However, we will also need a version of duality arising from the pairing $\bs{T}\otimes_\Lambda\bs{T}^c\to\Lambda(1)$. This runs into subtleties at ramified places away from $p$, since the unramified local conditions are no longer in perfect pairing, and we encounter error terms which contribute to Tamagawa numbers. More precisely, \cite[\S 7.8.4.4]{NekovarSC} gives an exact triangle
\begin{equation}\tag{$\vee$}
  \widetilde{\RG}_{\spadesuit,\clubsuit}(\CM,\bs{T})\to\RHom_\Lambda(\widetilde{\RG}_{\clubsuit,\spadesuit}(\CM,\bs{T}),\Lambda)[-3]\to\bigoplus_{v\in S-\{\p,\bar{\p}\}}\mathrm{Err}_v\to [1]
\end{equation}
The error term above is the definition in \cite[Section 6.2.3]{NekovarSC} applied to the unramified local conditions. The necessary properties for our purposes are stated in the following proposition.
\begin{prop}\label{prop:Err}
  Let $P$ be a height one prime ideal of $\Lambda$, then
  \[
    \mathrm{length}_{\Lambda_P}(\h^i(\mathrm{Err}_v)_P)=\begin{cases}
      \mathrm{Tam}_v(\bs{T},P) & i=1,2\\
      0 & \mathrm{otherwise}
    \end{cases}
  \]
  where the Tamagawa number is defined by
  \[
    \mathrm{Tam}_v(\bs{T},P)=\mathrm{length}_{\Lambda_P}(\h^1(I_v,\bs{T}_P)^{\Fr_v=1}_\mathrm{tors})
  \]
\end{prop}
\begin{proof}
  This follows from \cite[\S 7.6.10.7]{NekovarSC} since $\Lambda$ is regular. The notation $\mathrm{Err}_v^\mathrm{ur}(\mathscr{D},T)$ used there is exactly our $\mathrm{Err}_v$.
\end{proof}
In particular, the error terms are 0 when localized at a minimal prime, so it does not enter into rank calculations. In particular, the next proposition, usually credited to Greenberg or Wiles, follows easily from global duality and the Euler--Poincar\'e characteristic formula. We include the proof for completeness.
\begin{prop}\label{prop:Wiles}
  Let $R$ be a Noetherian regular local domain. Let $V$ be a projective $R$-module with an admissible action of $\Gal_\CM$, and let $V^\vee(1)=\Hom_R(V,R(1))$. Suppose $V^{\Gal_\CM}=(V^\vee(1))^{\Gal_\CM}=0$.
  
  For each $v\in S$, choose a $\Gal_{\CM_v}$-stable subspace $V_v\subseteq V$ and define the local condition $\Delta_v$ to be the Greenberg condition corresponding to it. Let $\Delta^\vee$ be the local condition for $V^\vee(1)$ which is the orthogonal complement to $\Delta$, then
  \[
    \rank_R\widetilde{\h}^1_{\Delta}(\Gal_{\CM,S},V)-\rank_R\widetilde{\h}^1_{\Delta^\vee}(\Gal_{\CM,S},V^\vee(1))=-\rank_R V+\sum_{v\in S}[\CM_v:\Q_v]\rank_R V_v
  \]
\end{prop}
\begin{proof}
  For each $v\in S$, we have $U_v^+=C^\bullet_\cont(\Gal_{\CM_v},V_v)$. Let $\mathtt{str}$ denote the strict local condition where $U_v^+=0$ for each $v\in S$, then $\widetilde{\RG}_\mathtt{str}$ computes the compactly supported Galois cohomology in the sense of \cite[Definition 5.3.1.1]{NekovarSC}, and we have the following exact triangle
  \[
    \widetilde{\RG}_\mathtt{str}\to\widetilde{\RG}_\Delta\to\bigoplus_{v\in S}U_v^+\to [1]
  \]
  If $\chi(-)$ denotes the Euler--Poincar\'e characteristics, then we get
  \[
    \chi(\widetilde{\RG}_\Delta)=\chi(\widetilde{\RG}_\mathtt{str})+\sum_{v\in S}\chi(U_v^+)=\rank_R V-\sum_{v\in S}[\CM_v:\Q_v]\rank_R V_v
  \]
  where we have applied Theorems 5.3.6 and 5.2.11 of \cite{NekovarSC} to each of the two terms. Moreover, equation ($\vee$) and the remark after it shows that $\rank_R\widetilde{\h}^i_\Delta=\rank_R\widetilde{\h}^{3-i}_{\Delta^\vee}$. In particular,
  \begin{align*}
    \rank_R\widetilde{\h}^1_\Delta-\rank_R\widetilde{\h}^1_{\Delta^\vee}&=\rank_R\widetilde{\h}^1_\Delta-\rank_R\widetilde{\h}^2_\Delta\\
    &=-\chi(\widetilde{\RG}_\Delta)+\rank_R \widetilde{\h}^0_\Delta
  \end{align*}
  Finally, observe that $\widetilde{\h}^0_\Delta(\Gal_{\CM,S},V)\subseteq V^{\Gal_\CM}=0$, so the two error terms are 0.
\end{proof}

In what follows, we will typically drop the $(\CM,\bs{T})$ from the notation when no confusion arises this way. We now study some basic properties of Selmer groups and Selmer complexes. In our applications, we will typically have the following mild hypothesis
\begin{equation}\tag{no-pole}
  \bs{T}^{\Gal_\CM}=0,\text{ and } \bs{T}_{\Gal_\CM}\text{ is pseudo-null}
\end{equation}
Both statements will usually hold because the Galois action is sufficiently non-trivial. Assuming this, it is immediately seen that $\widetilde{\h}^0_{\spadesuit,\clubsuit}=0$, and ($\ast$) then implies $\widetilde{\h}^3_{\spadesuit,\clubsuit}$ is pseudo-null. They are also related to torsions in $\widetilde{\h}^1$, as shown in the next proposition.
\begin{prop}\label{prop:Tors}
  If (no-pole) holds, then the torsion subgroup of $\widetilde{\h}^1_{\spadesuit,\clubsuit}(\CM,\bs{T})$ is pseudo-null.
\end{prop}
\begin{proof}
  In equation ($\vee$), let $A^1$ be the first cohomology of the middle term, then a spectral sequence argument plus the fact that $\Hom_\Lambda(-,\Lambda)$ is always torsion-free shows that $(A^1)_\tors=\Ext^1_\Lambda(\widetilde{\h}^3_{\clubsuit,\spadesuit}(\CM,\bs{T}),\Lambda)$. By the discussion before this proposition, the first term inside $\Ext^1_\Lambda$ is pseudo-null, so it has codimension at least 2. Recall from homological algebra that $\Ext^i_\Lambda(M,\Lambda)=0$ if $i<\codim(M)$ \cite[Theorem 9.1.3(iii)]{NekovarSC}. Therefore, $(A^1)_\mathrm{tors}=0$. Finally, the long exact sequence attached to ($\vee$) has the form
  \[
    A^0\to\mathrm{Err}^0\to\widetilde{\h}^1_{\spadesuit,\clubsuit}(\CM,\bs{T})\to A^1
  \]
  By Proposition~\ref{prop:Err}, $\mathrm{Err}^0$ is pseudo-null, and the desired result follows.
\end{proof}
For each $v$, let $U_v^-=\operatorname{Cone}(-\Delta_v)$, then by definition, we have an exact triangle
\[
  \widetilde{\RG}_{\spadesuit,\clubsuit}\to\RG(\Gal_{\CM,S},\bs{T})\to\bigoplus_{v\in S}U_v^-\to[1]
\]
If $\Delta_v$ is the unramified local condition, then $U_v^-$ is concentrated in degrees $[1,2]$. If $\Delta_v$ is the Greenberg local condition defined by $\bs{T}^?$, then $U_v^-=C_\cont^\bullet(\Gal_{\CM_v},\bs{T}/\bs{T}^?)$, so it is concentrated in degrees $[0,2]$. Taking cohomology gives the following exact sequences: for $\widetilde{\h}^0$
\[
  0\to\widetilde{\h}^0_{\spadesuit,\clubsuit}\to\bs{T}^{\Gal_\CM}\to\big(\bs{T}/\bs{T}^\spadesuit\big)^{\Gal_{\CM_\p}}\oplus (\bs{T}/\bs{T}^\clubsuit\big)^{\Gal_{\CM_{\bar\p}}}
\]
and for $\widetilde{\h}^1$
\[
  \big(\bs{T}/\bs{T}^\spadesuit\big)^{\Gal_{\CM_\p}}\oplus (\bs{T}/\bs{T}^\clubsuit\big)^{\Gal_{\CM_{\bar\p}}}\to\widetilde{\h}^1_{\spadesuit,\clubsuit}\to \h^1_{\spadesuit,\clubsuit}\to 0
\]
To study $\widetilde{\h}^2$, we use global duality in the form of equation ($\ast$) to obtain the following exact sequence
\[
  0\to X_{\clubsuit,\spadesuit}\to\widetilde{\h}^2_{\spadesuit,\clubsuit}\to \h^2(\CM_\p,\bs{T}^\spadesuit)\oplus \h^2(\CM_{\bar\p},\bs{T}^\clubsuit)
\]
The error terms with local invariants and co-invariants are related to trivial zeroes, cf.\ the computation in \cite[Section 0.10]{NekovarSC}. From an algebraic point of view, this is an error term which shows up in the control theorem. We will often impose the following hypothesis to rule them out at the family level.
\begin{equation}\tag{no-triv-zero}
  \big(\bs{T}/\bs{T}^\spadesuit\big)^{\Gal_{\CM_\p}}=0,\ \spadesuit\in\{\an,\bn\}
\end{equation}
Assuming it and (no-pole), we have $\widetilde{\h}^1_{\spadesuit,\clubsuit}=\h^1_{\spadesuit,\clubsuit}$, and $X_{\clubsuit,\spadesuit}=\widetilde{\h}  ^2_{\spadesuit,\clubsuit}$ up to a pseudo-null quotient. In the later applications in \S\ref{sec:Guess}, the module $\bs{T}$ will come with a bi-filtration whose graded pieces are projective $\Lambda$-modules with non-trivial Galois action, which implies (no-triv-zero). 

\subsection{Equivalence of main conjectures}
In classical Iwasawa theory, we want to study the compact $\mathbb{I}$-modules $X_{\bn,\ad}$ and $X_{\an,\bd}$ since they interpolate Bloch--Kato Selmer groups if $\bs{T}^\an$ and $\bs{T}^\bn$ are chosen appropriately. They are singled out in parts because their local conditions are self-dual. We might expect the central values of the associated $L$-functions to not vanish generically in the family we are considering, in which case a module like $X$ would be torsion, and we have an Iwasawa main conjecture equating its characteristic ideal to a $p$-adic $L$-function, cf.~\cite{GreenbergMotives}.

In the case which inspired the current axiomatization, there is a further feature, namely the $L$-functions are self-dual, and the sign of the function equation is constant in family. In the sign $-1$ case, the central values vanish automatically, and we instead expect $X$ to have rank 1. This is the case Perrin-Riou considered in \cite{PRHeegnerPoint}, leading to her Heegner point main conjecture.

In our setting, $X_{\an,\bd}$ and $X_{\ad,\bn}$ will correspond to families with opposite signs. Our main result of this section is that a good global class can be used to relate the Iwasawa main conjecture and the Perrin-Riou main conjecture. To make this precise, we need to introduce the regulator map.
\begin{definition}
  The regulator map is the following composite
  \[
    \mathrm{reg}:\widetilde{\h}^1_{\bn,\ad}\to \h^1(\CM_\p,\bs{T}^\bn)\to \h^1(\CM_\p,\bs{T}^\bn/\bs{T}^\an)\xrightarrow{\,\mathrm{Col}\,}\bs{M}
  \]
  where $\bs{M}$ is abstractly isomorphic to $\bs{T}^\bn/\bs{T}^\an$, so by (ord), it is a rank-1 $\mathbb{I}$-module which is projective over $\Lambda$.
\end{definition}

\begin{theorem}\label{thm:Equiv}
  Let $\bs{T}$ be an ordinary conjugate self-dual family of Galois representations. Suppose there exists an element $\bs{z}\in\widetilde{\h}^1_{\bn,\ad}$ such that
  \begin{enumerate}
    \item $\widetilde{\h}^1_{\bn,\ad}/\mathbb{I}\cdot\bs{z}$ is a torsion $\mathbb{I}$-module.
    \item $\mathrm{reg}(\bs{z})$ is non-torsion.
  \end{enumerate}
  then $\rank_\mathbb{I}\widetilde{\h}^2_{\bn,\ad}=1$, $\rank_\mathbb{I}\widetilde{\h}^2_{\an,\bd}=0$. Moreover, if \textup{(no-pole)} holds, then the following equivalence holds
  \[
    \Char_\Lambda\widetilde{\h}^2_{\an,\bd}=\Char_\Lambda(\mathrm{reg}(\bs{z}))^2\iff \Char_\Lambda(\widetilde{\h}^2_{\bn,\ad})_\mathrm{tors}=\Char_\Lambda\left(\frac{\widetilde{\h}^1_{\bn,\ad}}{\mathbb{I}\cdot{\bs{z}}}\right)^2
  \]
  In fact, the equalities can be replaced by either divisibility in the above equivalence.
  
  In particular, if in addition \textup{(no-triv-zero)} holds, then we can replace $\widetilde{\h}^1$ by $\h^1$ and $\widetilde{\h}^2$ by $X$ in the above statements to get an equivalence between two types of Iwasawa main conjectures.
\end{theorem}

\begin{proof}
For $\spadesuit\in\{\an,\bn\}$, the Selmer complex $\widetilde{\RG}_{\spadesuit,\spadesuit}$ is self-dual, so we have $\rank_\mathbb{I}\widetilde{\h}^1_{\spadesuit,\spadesuit}=\rank_\mathbb{I}\widetilde{\h}^2_{\spadesuit,\spadesuit}$. Our hypothesis therefore implies $\rank_\mathbb{I}\widetilde{\h}^2_{\bn,\ad}=1$. We have the exact triangle
\[
  \widetilde{\RG}_{\an,\ad}\to\widetilde{\RG}_{\bn,\ad}\to U_{\p}\to[1]
\]
where $U_\p:=C_\cont^\bullet(\Gal_{\CM_\p},\bs{T}^\bn/\bs{T}^\an)$. By Corollary~\ref{cor:Coleman}, $U_{\p}\simeq\bs{M}[1]$ in $D^b(\mathbf{psMod}_\Lambda)$. The associated long exact sequence contains the following piece
\[
  (\bs{T}^\bn/\bs{T}^\an)^{\Gal_{\CM_\p}}\to\widetilde{\h}^1_{\an,\ad}\to\widetilde{\h}^1_{\bn,\ad}\to \h^1(\CM_\p,\bs{T}^\bn/\bs{T}^\an)
\]
By the assumption on the Galois action, the first term is trivial, and the Coleman map gives an isomorphism of the final term with a rank one $\mathbb{I}$-module. Therefore, $\rank_\mathbb{I}\widetilde{\h}^1_{\bn,\ad}\leq\rank_\mathbb{I}\widetilde{\h}^1_{\an,\ad}+1$. Moreover, assumption (2) implies the final morphism is surjective when tensored with $\Frac(\mathbb{I})$, so equality holds and $\rank_\mathbb{I}\widetilde{\h}^1_{\an,\ad}=0$.

On the other hand, we can apply Proposition~\ref{prop:Wiles} to get
\[
  \rank_\Lambda\widetilde{\h}^1_{\bn,\bd}-\rank_\Lambda\widetilde{\h}^1_{\an,\ad}=\rank_\Lambda\bs{T}^\bn+\rank_\Lambda\bs{T}^{c,\bd}-\rank_\Lambda\bs{T}
\]
We can further convert all $\rank_\Lambda$ to $\rank_\mathbb{I}$ since $\Lambda\to\mathbb{I}$ is finite. By assumption, $\bs{T}^\bd$ and $\bs{T}^{c,\an}$ are mutual annihilators, so $\rank_\mathbb{I}\bs{T}^\bd+\rank_\mathbb{I}\bs{T}^{c,\an}=\rank_\mathbb{I}\bs{T}$. Again by ordinarity, $\rank_\mathbb{I}\bs{T}^\bn-\rank_\mathbb{I}\bs{T}^\an=1$. Therefore,
\[
  \rank_\mathbb{I}\widetilde{\h}^1_{\bn,\bd}-\rank_\mathbb{I}\widetilde{\h}^1_{\an,\ad}=1
\]
It follows that $\rank_\mathbb{I}\widetilde{\h}^1_{\bn,\bd}=1$. By applying the argument in the previous paragraph to the exact triangle $\widetilde{\RG}_{\an,\bd}\to\widetilde{\RG}_{\bn,\bd}\to U_{\p}\to[1]$, it remains to show that $\widetilde{\h}^1_{\bn,\bd}\to \h^1(\CM_\p,\bs{T}^\bn/\bs{T}^\an)$ has non-torsion image. This again follows from assumption (2) since the regulator map factors through $\widetilde{\h}^1_{\bn,\bd}$.

We now verify the equivalence in the second part of the statement. First observe that by the rank calculation above and Proposition~\ref{prop:Tors}, we must have $\widetilde{\h}^1_{\an,\ad}$ is pseudo-null. Let $C=\coker(\mathrm{reg})$. It is torsion since $\bs{M}$ has $\mathbb{I}$-rank 1 and $\Im(\mathrm{reg})$ contains a non-torsion element by part (2) of the hypothesis. The long exact sequence attached to the exact triangle $\widetilde{\RG}_{\an,\ad}\to\widetilde{\RG}_{\bn,\ad}\to U_{\p}\to[1]$ gives
\[
  \widetilde{\h}^1_{\an,\ad}\to\widetilde{\h}^1_{\bn,\ad}\to\bs{M}\to\widetilde{\h}^2_{\an,\ad}\to\widetilde{\h}^2_{\bn,ad}\to \h^2(U_\p)
\]
The two end terms are pseudo-null. In the middle, the map $\widetilde{\h}^1_{\bn,\ad}\to\bs{M}$ is exactly the regulator map by definition. Therefore, we can split it into the following two short exact sequences:
\begin{align*}
  \text{(pseudo-null)}&\to\widetilde{\h}^1_{\bn,\ad}\to\bs{M}\to C\to 0 \\
  0&\to C\to\widetilde{\h}^2_{\an,\ad}\to\widetilde{\h}^2_{\bn,\ad}\to\text{(pseudo-null)}
\end{align*}
In the first sequence, after quotienting out by $\mathbb{I}\cdot\bs{z}$, we get the relation
\[
  \Char_\Lambda(C)\Char_\Lambda\frac{\widetilde{\h}^1_{\bn,\ad}}{\mathbb{I}\cdot\bs{z}}=\Char_\Lambda(\mathrm{reg}(\bs{z}))
\]
In the second sequence, since $C$ is torsion, we can take the torsion subgroup without affecting the pseudo-null cokernel, so
\[
  \Char_\Lambda\big(\widetilde{\h}^2_{\an,\ad}\big)_\tors=\Char_\Lambda C\Char_\Lambda\big(\widetilde{\h}^2_{\bn,\ad}\big)_\tors
\]
Again by Proposition~\ref{prop:Tors},  the module $C$ is also the cokernel of localization from $\widetilde{\h}^1_{\bn,\bd}$, so by the same argument as before, we have a third exact sequence
\[
  0\to C\to\widetilde{\h}^2_{\an,\bd}\to\widetilde{\h}^2_{\bn,\bd}\to\text{(pseudo-null)}
\]
This gives another equality of characteristic ideals. Combining it with the previous two, we see that it remains to show that $\Char_\Lambda\big(\widetilde{\h}^2_{\an,\ad}\big)_\tors=\Char_\Lambda\widetilde{\h}^2_{\bn,\bd}$.

This is now a consequence of global duality. We work one prime at a time. Let $P$ be a height one prime in $\Lambda$, and let $R=\Lambda_P$. It is a discrete valuation ring, and we need to prove $(\widetilde{\h}^2_{\an,\ad})_\tors$ and $\widetilde{\h}^2_{\bn,\bd}$ have the same lengths after localizing at $P$. Since localization is exact, equation ($\vee$) still holds with $\Lambda$ replaced by $R$ everywhere. We apply it to get the following exact triangle
\[
  \widetilde{\RG}_{\an,\ad}\to\RHom_R(\widetilde{\RG}_{\bn,\bd},R)[3]\to\mathrm{Err}\to[1]
\]
Let $A^\bullet$ denote the cohomology of the middle term. It can be computed by the spectral sequence
\[
  \Ext^i_R(\widetilde{\h}^j_{\bn,\bd},R)\Rightarrow A^{i-j+3}
\]
Since $R$ is a DVR, the Ext-groups vanish for $i\geq 2$, and $\Ext^1(-,R)$ computes the torsion module. Moreover, we have computed that $\rank_R\widetilde{\h}^j_{\bn,\bd}=1$ if $j=1$ and is $0$ otherwise. The assumption (no-pole) implies $\widetilde{\h}^0_{\bn,\bd}=0$, and by global duality ($\ast$) also implies $\widetilde{\h}^3_{\bn,\bd}=0$. Finally, Proposition~\ref{prop:Tors} implies that $\widetilde{\h}^1_{\bn,\bd}$ is torsion-free, so it is abstractly isomorphic to $R$. The spectral sequence therefore degenerates on page 2 and gives
\[
  A^i=\begin{cases}
    R\oplus\widetilde{\h}^2_{\bn,\bd} & i=2\\
    0 & i\neq 2
  \end{cases}
\]
Plugging this into the long exact sequence and splitting it gives the following two short exact sequences
\begin{align*}
  0\to\mathrm{Err}^1\to\widetilde{\h}^2_{\an,\ad}&\to N\to 0\\
  0&\to N\to R\oplus\widetilde{\h}^2_{\bn,\bd}\to\mathrm{Err}^2\to 0
\end{align*}
By Proposition~\ref{prop:Err}, the two error terms are torsion and have the same length, so $N\simeq R\oplus((\widetilde{\h}^2_{\an,\ad})_\tors/\mathrm{Err}^1)$. But from the second sequence, $N\simeq R\oplus\ker(\widetilde{\h}^2_{\bn,\bd}\to\mathrm{Err}^2)$. Taking the length of $N_\tors$ gives the desired equality.
\end{proof}

\begin{remark}
  \begin{enumerate}
  \item We can interchange $\an$ with $\bn$ throughout the entire theorem. This requires localizing at $\bar{\p}$ instead of $\p$, and the proof goes through with this change. 
  \item If we know that $\mathbb{I}$ is Gorenstein and $\bs{T}$ is projective over $\mathbb{I}$, then we can do the above computations over $\mathbb{I}$ instead of $\Lambda$. This has the effect of taking all characteristic ideals over $\mathbb{I}$, which gives a refinement of the current statements.
  \item We could replace the base ring $\Lambda$ by any of its localizations without changing any argument, producing results over the localization. This is useful when the constructed lattice $\bs{T}$ has defects. For example, in \S\ref{ss:BigGaloisRep}, the pseudocharacter argument only shows conjugate self-duality over an affinoid subset of the weight space.
  \item We do not expect an implication of the form $\rank_\Lambda\widetilde{\h}^2_{\an,\bd}=0\implies \rank_\Lambda\widetilde{\h}^2_{\bn,\ad}=1$ in this generality. Running through the argument would give a local class in $\widetilde{\h}^1_{\bn}(\CM_\p,\bs{T})$, but it is difficult to show that this comes from a global class.
  \end{enumerate}
\end{remark}

\subsection{Applications}\label{sec:App}
In this subsection, we give two applications of the above formalism. The first will recover one main theorem of \cite{CH18} without the rank 0 Kolyvagin system argument. The second allows one to deduce the rank 0 results in \cite{CastellaTuan} using the Euler systems they constructed. This section also hints at the wall crossing phenomena we will explain in general in the following section.

\subsubsection{Set-up}
We specialize to $F=\Q$, though many of the results stated here extend (or should extend) to more general totally real fields.

Let $N$ be a square-free natural number. Let $f\in S_k^\mathrm{new}(\Gamma_0(N))$ be a modular form which is ordinary at $p$. The modular form $f$ has an attached $p$-adic Galois representation $T_f:\Gal_\Q\to\GL_2(\Or)$, where $\Or$ is a finite extension of $\Z_p$. We assume that $T_f$ is residually irreducible. Let $D_p$ be a decomposition group of $\Gal_\Q$ at $p$, then $T_f$ has the following $D_p$-stable filtration
\[
  0\to T_f^-\to T_f\to T_f^+\to 0
\]
where $D_p$ acts on $T_f^-$ as $\alpha_p\chi_\mathrm{cyc}$ and on $T_f^+$ as $\beta_p$, for some units $\alpha_p,\beta_p\in\Or^\times$, identified as the unramified character sending Frobenius to that value. Both $T_f^-$ and $T_f^+$ are 1-dimensional, and there is a perfect pairing on $T_f$ exchanging $T_f^-$ and $T_f^+$.

We are interested in the anticyclotomic Iwasawa theory of $f$. Fix an anticyclotomic character $\chi$ of our imaginary quadratic field $\CM$. For simplicity, we assume the numbers $N\ ,\mathrm{disc}(\CM),\ \mathrm{cond}(\chi),\ p$ are pairwise coprime. Let $\CM^\mathrm{ac}$ be the anticyclotomic $\Z_p$-extension of $\CM$, $\Gamma^\mathrm{ac}=\Gal(\CM^\mathrm{ac}/\CM)\simeq\Z_p$, $\Lambda^\mathrm{ac}=\Or[[\Gamma^\mathrm{ac}]]$, and $\Psi:\Gal_\CM\to\Lambda^\mathrm{ac}$ be the canonical character. Our $p$-adic family of Galois representation will be
\[
  \bs{T}=T_f\otimes_{\Or}\Lambda^\mathrm{ac}(\chi\Psi)
\]
with $\Gal_\CM$ acting diagonally. This certainly satisfies (no-pole) by residual irreducibility. For the Selmer condition defined by either $\bs{T}$ or $T_f^-\otimes_{\Or}\Lambda^\mathrm{ac}(\chi\Psi)$, the condition (no-triv-zero) is easily checked. Moreover, the pairing on $T_f$ induces a perfect pairing on $\bs{T}$.

\subsubsection{Heegner points}\label{ss:Heegner}
The theory in this setting was very clearly set up in \cite{BCK}, and this was the model on which we built our abstract framework. Suppose the generalized Heegner condition holds, namely the number of primes $\ell|N$ which is inert in $\CM$ is even. Let $\bs{T}^\an=0$ and $\bs{T}^\bn=T_f^-\otimes_{\Z_p}\Lambda^\mathrm{ac}(\chi\Psi)$. This satisfies (ord) since $\chi$ is unramified at $p$. The dual spaces are $\bs{T}^{c,\ad}=\bs{T}^\bn$ and $\bs{T}^{c,\bd}=\bs{T}$.

The special class here is the Howard's big Heegner point \cite{HowardGL2}.
\[
  \bs{z}_c\in \varprojlim_n \h^1(\CM[cp^n],T_f^-)=\h^1(\CM[c],T_f^-\otimes\Lambda^\mathrm{ac}(\Psi))
\]
where we have used the notation $\CM[c]$ to denote the ring class field of $\CM$ with conductor $c$, and the equality is by a standard application of Shapiro's lemma. If $c=\mathrm{cond}(\chi)$, then we can further project to the $\chi$-isotypic part to produce an element $\bs{z}^\chi\in \h^1(\CM,\bs{T}^\bn)$. Since the dual condition is also defined by the same space, we have $\bs{z}^\chi\in \h^1_{\bn,\ad}$.

To verify the hypothesis (2) of Theorem~\ref{thm:Equiv}, we need to compute $\mathrm{reg}(\bs{z}^\chi)$, where $\bs{z}^\chi\in\varprojlim_n \h^1(\CM[p^n]_\p,T_f^-)$. This is the explicit reciprocity law proven in \cite[Theorem 4.4]{BCK}. It now follows from \cite{BurungaleHeegII} that the localization at $\p$ is non-torsion.

Hypothesis (1) as well as one divisibility of the Perrin-Riou main conjecture are both proven in \cite{HowardGL2} using a bipartite Euler system argument. Using our theorem, this implies
\[
  \Char_{\Lambda^\mathrm{ac}}X_{\an,\bd}|(\sh{L}_p^\mathrm{BDP})^2
\]
In classical terms, the Selmer group $X_{\an,\bd}$ is the Greenberg Selmer group with the strict and relaxed conditions at $\p$ and $\bar\p$ respectively. The other side (with the square) is the $p$-adic $L$-function first constructed by Bertolini--Darmon--Prasanna \cite{BDP} interpolating the algebraic parts of the central values $L(\frac{1}{2},f\times\chi')$ where the weights of $\chi'$ are greater than those of $f$. It follows by a standard application of the control theorem that $L(\frac{1}{2},f\times\chi')\neq 0$ implies $\mathrm{Sel}_\mathrm{Gr}(V_f\otimes\chi')=0$ under the previous assumptions on $f$.

\begin{remark}
  By applying complex conjugation to our argument, we can prove a similar rank-0 results for the relaxed-strict Selmer group, corresponding to the case where the weight of $\chi'$ is less than 0. In this case, the result is just the conjugate of the one for positive weight. However, we will see that in general, one special cycle can control several rank 0 regions which appear to be fundamentally different.
\end{remark}

\subsubsection{No Heegner points}\label{ss:NoHeegner}
We now consider the opposite of the above setting, so the number of prime $\ell|N$ which is inert in $\CM$ is odd. In this case, we know that $L(\frac{1}{2},f\times\chi)\neq 0$ for a generic finite order $\chi$ by the main result of \cite{Vatsal03}. They can be interpolated into a $p$-adic $L$-function $\sh{L}_p^{\mathrm{BD}}$ (cf.~\cite[Theorem 5.1.1]{CastellaTuan}), and we have an Iwasawa main conjecture of the form
\[
  X_\mathrm{ord}(E/\CM)|(\sh{L}_p^{\mathrm{BD}})
\]
which implies a generic rank 0 result. This conjecture was established by Bertolini--Darmon \cite{BD-IMC} using level-raising methods under mild technical hypotheses.

In \cite{CastellaTuan}, Castella-Tuan constructed an Euler system for the representation $V_f\otimes\chi'$, where $\chi'$ has large weight. In short, they specialized the Euler system constructed by Bertolini--Seveso--Venerucci \cite{BSV1} in the balanced triple product setting to the case where two of the forms have CM by $\CM$. As was further shown in \cite{CastellaTuan}, the explicit reciprocity law computes the big logarithm of their class as a $p$-adic $L$-function with certain interpolation properties coming from an easy factorization of the unbalanced triple product $p$-adic $L$-function.

In our formalism, let $\bs{T}^{\an}=T_f^-\otimes\Lambda^\mathrm{ac}(\chi\Psi)$ and $\bs{T}^{\bn}=\bs{T}$, so we start with information about the Greenberg Selmer group and want to derive a rank-0 result for the usual Selmer group. The above construction gives rise to a class $\bs{z}_\mathrm{CT}$ satisfying the following properties
\[
  \bs{z}_\mathrm{CT}\in \h^1_{\bn,\ad},\quad \mathrm{reg}(\bs{z}_\mathrm{CT})^2=\sh{L}_p^\mathrm{BD}
\]
The interpolation property as well as the result of Vatsal cited earlier imply that $\sh{L}_p^\mathrm{BD}$ is non-torsion, so Jetchev--Nekov\'a\v{r}--Skinner's Euler system argument \cite{JNS} implies one divisibility of the Perrin-Riou main conjecture. Therefore, by Theorem~\ref{thm:Equiv},
\[
  \Char_{\Lambda^\mathrm{ac}}X_\mathrm{ord}|\sh{L}_p^\mathrm{BD}
\]
The usual control theorem argument implies a rank 0 result. This was deduced without appealing to cyclotomic constructions such as the Rankin--Eisenstein elements of \cite{KLZ1}. While this is not known at the moment, the generalization of the results in the triple product setting to the totally real field case appears to be within reach, so this argument could lead to rank-0 results in those cases as well, where the cyclotomic constructions are not known.

\section{Automorphic background}\label{sec:Aut}
Before presenting our proposed framework, we use this expository section to recall certain aspects of the Gan--Gross--Prasad conjecture. In the discrete series case, we will also explain the combinatorial recipe relating the distinguished signatures and the weight interlacing conditions.

\subsection{Discrete series}\label{ss:DiscreteSeries}
This section recalls some facts about the discrete series on real unitary groups. Throughout this section, the group $\U(p,q)$ is the unitary group of the Hermitian space $V_{p,q}$ defined using the Hermitian form $\mathbf{1}_{p,q}=\big(\begin{smallmatrix}\mathbf{1}_p&\\&-\mathbf{1}_q\end{smallmatrix}\big)$. We view it as a real Lie group embedded in $\GL_n(\C)$, where $n=p+q$.

\subsubsection{Root data}
Let $T$ be the diagonal torus in $\U(p,q)$, then $T\simeq\U(1)^n$ is a compact maximal torus. Define the characters $t_1,\cdots,t_n$, so that $z\in T$ is equal to $\diag(t_1(z),\cdots,t_n(z))$. We will label a general character $\sum_{i=1}^n a_i t_i$ by the tuple $\underline{a}=(a_1,\cdots,a_n)\in\Z^n$. We pick $\Delta^+=\{t_i-t_j\,|\,1\leq i<j\leq n\}$ to be the set of positive roots, and we let
\[
  \rho=\left(\frac{n-1}{2},\frac{n-3}{2},\cdots,\frac{1-n}{2}\right)
\]
be the half sum of positive roots.

Let $K_{p,q}\subseteq\U(p,q)$ be the maximal compact subgroup consisting of matrices of the form $\big(\begin{smallmatrix} g&\\&h\end{smallmatrix}\big)$, where $g\in\U(p)$ and $h\in\U(q)$. The torus $T$ is also a maximal torus of $K_{p,q}$.

\subsubsection{Harish-Chandra classification}\label{ss:HarishChandra}
The discrete series of the group $G=\U(p,q)$ can be indexed its Harish-Chandra parameter, which we combinatorially represent by
\begin{itemize}
  \item a regular infinitesimal character $\underline{a}=(a_1>\cdots>a_n)\in\big(\Z+\frac{n-1}{2}\big)^n$ and
  \item a binary string $\epsilon\in\{0,1\}^n$ with exactly $p$ 0s.
\end{itemize}
We call the binary string its \emph{Harish-Chandra code}. Let $i_1,\cdots,i_p$ be the location of the 0s in $\epsilon$ and $j_1,\cdots,j_q$ be the location of the 1s, then $(\underline{a},\epsilon)$ represents
\[
  (a_{i_1},\cdots,a_{i_p};a_{j_1},\cdots,a_{j_q})\in X^*(T)+\rho
\]
in the usual notations for the Harish-Chandra parameter. In particular, the holomorphic discrete series whose minimal $K$-type is $\underline{\kappa}=(\kappa_1\geq\cdots\geq\kappa_p;\kappa_{p+1}\geq\cdots\geq\kappa_n)$ is represented by the pair
\[
  (\underline{\kappa}+\rho,\underbrace{1\cdots1}_q\underbrace{0\cdots0}_p)
\]
Such a representation exists if and only if $\kappa_n\geq\kappa_1+n$. This will be generalized in \S\ref{ss:CoherentCoh}.

\subsubsection{Generic datum}
We now specify the choices needed to call a representation ``generic''. This will be used in the following section to label elements of the Vogan $L$-packet.

Let $V$ be a split Hermitian space over $\C$. Let $G=\U(V)$, then $G$ is quasi-split. Let $B$ be a Borel subgroup of $G$ with Levi decomposition $B=T\ltimes N$. Let $\lambda:N\to\C^\times$ be a generic character, then a representation $\pi$ is $(N,\lambda)$-generic if
\[
  \dim_\C\Hom_{N}(\pi,\lambda)=1
\]
The torus $T$ acts on $\lambda$, and the above condition is invariant under such action. It is now useful to split into two cases, depending on the parity of $n=\dim V$.

\begin{description}
  \item[$n$ odd] The action of $T$ on the space of generic characters is transitive, and there is a unique generic discrete series representation for each infinitesimal character $\underline{a}$. If $n=2m+1$, then it is the representation on $\U(m+1,m)$ with Harish-Chandra code $01\cdots 010$.
  \item[$n$ even] The action of $T$ on the space of generic characters has two orbits, which correspond to the two $\R_+^\times$-orbits on additive characters $\psi:\C/\R\to\C^\times$ (cf.\ \cite[Section 12]{GGPConjecture}). A set of representative characters is
  \[
    \psi_\C:z\mapsto\exp(2\pi(\bar{z}-z)),\quad\overline{\psi_\C}:z\mapsto\exp(2\pi (z-\bar{z}))
  \]
  Let $n=2m$, then the Harish-Chandra code for the $\overline{\psi_\C}$-generic (resp.~$\psi_\C$-generic) representation is $01\cdots 01$ (resp.~$10\cdots 10$).
\end{description}
Proofs of the above claims can be found in \cite[\S A.3]{Atobe20}.

\subsubsection{Coherent cohomology}\label{ss:CoherentCoh}
Let $V$ be a Hermitian space defined over the extension $\CM/\sh{F}$. At the archimedean places, write its signature as $\{(a_\sigma,b_\sigma)\}_{\sigma\in\Sigma_\infty^+}$. Let $\mb{G}=\U(V)$ be its unitary group, then we fix an isomorphism
\[
  \mb{G}_\infty\simeq\prod_{\sigma\in\Sigma_\infty^+}\U(a_\sigma,b_\sigma).
\]
Let $K_\infty\subseteq\mb{G}_\infty$ be the maximal compact subgroup of $\mb{G}_\infty$ which corresponds to $\prod_\sigma K_{a_\sigma,b_\sigma}$ under this isomorphism. The Shimura datum for $\mb{G}$ defined by
\[
  z\mapsto\left(\begin{smallmatrix}z/\bar{z}\cdot\mathbf{1}_{a_\sigma} & \\ & \mathbf{1}_{b_\sigma}\end{smallmatrix}\right)_{\sigma\in\Sigma_\infty^+}
\]
gives rise to a Shimura variety $\Sh_{\mb{G}}$ whose reflex field is contained in $\CM$. Let $\xi$ be an algebraic representation of $K_\infty$, then it gives rise to an automorphic vector bundle $\sh{V}_\xi$ over $\Sh_{\mb{G}}$. 

Let $\pi$ be a cusp form on $\mb{G}$ whose archimedean components lie in the discrete series. For each $\sigma\in\Sigma_\infty^+$, we have a pair $(\underline{a}_\sigma,\epsilon_\sigma)$ as described in \S\ref{ss:HarishChandra}. Let $q_\sigma$ denotes the number of times where 0 occurs before 1 in $\epsilon_\sigma$, so $0\leq q_\sigma\leq a_\sigma b_\sigma$. Let $q_\pi=\sum_\sigma q_\sigma$, then $0\leq q_\pi\leq\dim\Sh_{\mb{G}}$.

Suppose $\underline{a}_\sigma$ is sufficiently far from walls for all $\sigma$, then by theorems B and D of \cite{Harris90}, $\pi$ contributes to the coherent cohomology of $\sh{V}_\xi$ for a unique $\xi$, in which case its contribution is concentrated in the degree $q_\pi$. We will use this to check some dimension compatibility for $p$-adic $L$-functions.

\subsection{Gan--Gross--Prasad conjecture}\label{sec:LocalGGP}
Let $\Pi=\Pi_{n-1}\otimes\Pi_n$ be a RACSDC automorphic representation on $\GL_n(\A_\CM)\times\GL_{n+1}(\A_\CM)$. We are interested in the arithmetic of $\Pi$, in particular its central $L$-value $L\big(\frac{1}{2},\Pi\big)$. This value depends on the endoscopic transfer of $\Pi$ to a specific pair of unitary groups, predicted by a local restriction problem. This section will briefly recall the key points of this picture.

\subsubsection{Vogan packets}
Let $v$ be a place of $\sh{F}$ which does not split, and let $w$ be the place of $\CM$ above $v$. For $i=n-1,n$, the local Langlands correspondence attaches to $\Pi_i$ an $L$-parameter
\[
  \phi_{i,w}:\mathrm{WD}_{\CM_w}\to\GL_i(\C).
\]
It is conjugate self-dual. The algebraicity requirement forces the sign to be $(-1)^{i-1}$ (cf.~\S\ref{ss:Arch}). It follows from \cite[Theorem 8.1]{GGPConjecture} that $\phi_i$ descends to an $L$-parameter
\[
  \phi_{i,v}:\mathrm{WD}_{\sh{F}_v}\to {}^L\U_i.
\]

We now explain some aspects of endoscopic classification following \cite[Section 9]{GGPConjecture}. By definition, the Vogan packet $\Pi_{\phi_i,v}$ attached to $\phi_{i,v}$ is a disjoint union
\[
  \Pi_{\phi_i,v}=\bigsqcup_{V}\Pi_{\phi_i,v}^{V},
\]
where $V$ runs over all isomorphism classes of Hermitian spaces over $\CM_w/\sh{F}_v$ of dimension $i$, and each $\Pi_{\phi_i,v}^V$ is a collection of irreducible admissible representations of $\U(V)$. Given $\pi\in\Pi_{\phi_i,v}$, let $V^{(\pi)}$ be the Hermitian space such that $\pi$ is a member of $\Pi_{\phi_i,v}^{V^{(\pi)}}$. For any representation $\pi\in\Pi_{\phi_i,v}$, its base change to $\CM_w$ is $\Pi_{i,w}$. 

Let $A_{\phi_i,v}$ be the component group of the centralizer of $\phi_{i,v}$. By the discussion in \cite[Section 4]{GGPConjecture}, it is an elementary abelian 2-group. There is a bijection
\[
  J_v:A_{\phi_i,v}^\vee\simeq\Pi_{\phi_i,v}.
\]
More precisely, this depends on the following auxiliary data.
\begin{description}
  \item[$i$ odd] There are exactly two split Hermitian spaces of dimension $i$ up to isomorphism. They are distinguished by their discriminants $\delta\in\sh{F}_v^\times/\Norm_{\CM_w/\sh{F}_v}\CM_w^\times$. We need to choose one of them.
  \item[$i$ even] The action of $\Norm_{\CM_w/\sh{F}_v}\CM_w^\times$ on the space of additive characters $\CM_w/\sh{F}_v\to\C^\times$ has two orbits. We need to choose one of the orbits.
\end{description}
Following \cite[Section 8]{GGPConjecture}, this choice determines a generic datum, which determines a unique generic representation in $\Pi_{\phi_i,v}$. By construction, it corresponds to the trivial character. The bijection is characterized by certain endoscopic character identities, which are established in \cite[Theorem 1.6.1]{KMSW}.

\subsubsection{Archimedean packets}\label{ss:Arch}
Let $v$ be an archimedean place of $\sh{F}$, and let $w$ be the place of $\CM$ above $v$. The CM type $\Sigma_\infty^+$ selected in the beginning gives an identification $\CM_w\simeq\C$.

Let $i\in\{n-1,n\}$. Since $\Pi_{i,w}$ is regular algebraic, its infinitesimal character $\underline{a}$ lies in $(\Z+\frac{i-1}{2})^i$. Its Langlands parameter is the representation $\varphi_{\underline{a}}:\C^\times\to\GL_n(\C)$ defined by
\[
  \varphi_{\underline{a}}(z)=\diag\big((z/\bar{z})^{a_1},\cdots,(z/\bar{z})^{a_n}\big),
\]
where $(z/\bar{z})^{\alpha}:=z^{2\alpha}(z\bar{z})^{-\alpha}$ if $\alpha\in\frac{1}{2}\Z$. This is conjugate self-dual of sign $(-1)^{i-1}$. The Vogan $L$-packet for $\varphi_{\underline{a}}$ is the disjoint union
\[
  \Pi_{\underline{a}}=\bigsqcup_{p=0}^n \Pi_{\underline{a}}^{(p)}
\]
where $\Pi_{\underline{a}}^{(p)}$ consists of all discrete series representations on $\U(p,q)$ whose infinitesimal character is $\underline{a}$. Let $\pi\in\Pi_{\underline{a}}$. Denote its Harish-Chandra code by $\epsilon_\pi^{\mathrm{HC}}\in\{0,1\}^n$. The association $\pi\mapsto\epsilon_\pi^{\mathrm{HC}}$ is a bijection.

The component group of $\varphi_{\underline{a}}$ is
\[
  A_n\simeq\prod_{i=1}^n(\Z/2\Z)e_i.
\]
Therefore, a character $\chi:A_n\to\{\pm 1\}$ can be represented by the binary string $\epsilon_\chi\in\{0,1\}^n$ such that $(\epsilon_\chi)_i=1$ if and only if $\chi(e_i)=-1$. From the previous subsection, we have a bijection
\[
  J_\infty:A_n^\vee\simeq\Pi_{\underline{a}}.
\]
Given $\pi\in\Pi_{\underline{a}}$, let $\epsilon_\pi^{\mathrm{End}}$ denote the binary string representing $J_\infty^{-1}(\pi)$. The relation between $\epsilon_\pi^{\mathrm{HC}}$ and $\epsilon_\pi^{\mathrm{End}}$ is very concrete. It is useful to split into two cases depending on the parity of $i$.

\begin{description}
  \item[$i$ odd] Let $i=2m+1$. For the split Hermitian space of dimension $2m+1$, we choose $V_{m+1,m}$, which has discriminant 1. In this case,
  \[
    \epsilon_\pi^{\mathrm{HC}}=\epsilon_\pi^{\mathrm{End}}+01\cdots 010
  \]
  \item[$i$ even] Let $i=2m$. The two orbits are represented by $\psi_\C$ and $\overline{\psi_\C}$. We choose the character $\overline{\psi_\C}$ (this will explained in the next subsection). In this case
  \[
    \epsilon_\pi^{\mathrm{HC}}=\epsilon_\pi^{\mathrm{End}}+01\cdots 01
  \]
\end{description}
Proofs of the above claims, namely verifications of the relevant endoscopic character identities, can be found in \cite[\S A.4]{Atobe20}. Note that in both cases, if $\pi$ is the generic representation corresponding to the \emph{opposite} choice, then $\epsilon_\pi^{\mathrm{End}}=11\cdots 1$.

\subsubsection{Local conjecture}
From the point of view of this paper, the local GGP conjecture identifies a specific element of the product $\Pi_{\phi_{n-1},v}\times\Pi_{\phi_n,v}$ which is important to the arithmetic of $\Pi$. We will recall this construction and make it explicit in the archimedean case.

Let $v$ be a place of $\sh{F}$ which does not split in $\CM$. Let $\psi_v:\CM_w/\sh{F}_v\to\C^\times$ be an additive character. If $M$ is any representation of the Weil--Deligne group $\mathrm{WD}_{\CM_w}$, then we have a local root number
\[
  \varepsilon(M,\psi_v)\in\{\pm 1\}
\]
defined in \cite[Section 5]{GGPConjecture}. Using it, we form two distinguished characters as follows. If $a\in A_{\phi_n,v}$, then the subspace $\phi_{n,v}^{a=-1}$ where $a$ acts by $-1$ is stable under $\mathrm{WD}_{\CM_w}$. Define
\[
  \chi_n:A_{\phi_n}\to\{\pm 1\},\ a\mapsto\varepsilon(\phi_{n-1,v}\otimes\phi_{n,v}^{a=-1},\psi_v)
\]
and similarly define $\chi_{n-1}:A_{\phi_{n-1}}\to\{\pm 1\}$. We can now state the local GGP conjecture, which was proven by Beuzart-Plessis \cite{BPLocalWeak,BPLocalRefined} in general.

\begin{theorem}
  Let $\mathtt{V}_{n-1}\subseteq\mathtt{V}_n$ be split Hermitian spaces over $\CM_w/\sh{F}_v$ of dimensions $n-1$ and $n$ respectively. There is a unique pair $(\pi_{n-1.v},\pi_{n,v})$ in $\Pi_{\phi_{n-1},v}\times\Pi_{\phi_n,v}$ satisfying the conditions
  \begin{enumerate}
    \item Relevancy: $V^{(\pi_{n-1,v})}\subseteq V^{(\pi_{n,v})}$, and the quotient is isomorphic to $\mathtt{V}_n/\mathtt{V}_{n-1}$;
    \item Distinguished: the Hom-space
    \[
      \Hom_{\U(V^{(\pi_{n-1,v})})}(\pi_n\otimes \pi_{n-1},\C)
    \]
    is non-trivial, which implies it is one dimensional.
  \end{enumerate}
  
  Let $\delta$ be the discriminant of the odd Hermitian space used to label the Vogan $L$-packet on the odd group. Let $\psi_v'(x)=\psi_v(-2\delta x)$. If we label the Vogan $L$-packet for the even space using $\psi_v'$, then $\pi_{i,v}=J_v(\chi_i)$ for $i=n-1,n$.
\end{theorem}

Now suppose $v|\infty$. We describe the distinguished representation explicitly in terms of their Harish-Chandra codes. In this case, the local conjecture was first proven by H.~He, who gave a different combinatorial recipe for describing the distinguished representation \cite{HeGGP}. Choose $\psi_\C(z)=e^{2\pi(\bar{z}-z)}$, and suppose the parameter $\phi$ corresponds to discrete series with infinitesimal character
\[
  \phi_{\underline{a}}\leftrightarrow(a_1>a_2>\cdots>a_n),\ \phi_{\underline{b}}\leftrightarrow(b_1>b_2>\cdots>b_{n-1})
\]
We consider the pair of parameters $(\phi_{\underline{b}}^\vee,\phi_{\underline{a}})$. This is the point of view of a restriction problem, which we find to be clearer. For $1\leq i\leq n$, we have $\phi_{\underline{a}}^{e_i=-1}=\C((z/\bar{z})^{a_i})$, so
\[
  \chi_n(e_i)=\varepsilon(\C((z/\bar{z})^{a_i}\otimes\phi_{\underline{b}}^\vee,\psi_\C)=\prod_{j=1}^{n-1}\varepsilon(\C((z/\bar{z})^{a_i-b_j}),\psi_\C)=(-1)^{\#\{j\,|\,a_i>b_j\}}
\]
where the final equality follows from \cite[(3.2.5)]{TateCorvallis}. Similarly, we have
\[
  \chi_{n-1}(f_j)=(-1)^{\#\{i\,|\,a_i>b_{n-j}\}}
\]
where the appearance of $n-j$ instead of $j$ is due to the contragredient on $\phi_{\underline{b}}$. The Harish-Chandra codes of the distinguished representations are obtained by adding $0101\cdots$ to the codes for $\chi_n$ and $\chi_{n-1}$.

Since we are always labelling the odd unitary group using the space $V_{m,m-1}$ of discriminant 1, our starting pair is one of the two pairs
\begin{equation}\label{eqn:RelevantPair}
  \U(m-1,m-1)\subseteq\U(m,m-1)\text{ or }\U(m,m-1)\subseteq\U(m,m),
\end{equation}
depending on the parity of $n$. The twist by $-2\delta$ in the statement of the theorem explains our previous choice of labelling the Vogan packets on the even unitary group using $\overline{\psi_\C}$. As a sanity check, we prove an easy combinatorial lemma which shows that the resulting Harish-Chandra codes always give a relevant pair of Hermitian spaces.

\begin{lemma}
  Let $\epsilon_n$ and $\epsilon_{n-1}$ be the binary strings obtained from the above procedure. Let $p_n$ (resp.~$p_{n-1}$) be the number of 1s in $\epsilon_n$ (resp.~$\epsilon_{n-1}$). If $n$ is even, then $p_n-p_{n-1}=1$, otherwise $p_n-p_{n-1}=0$.
\end{lemma}
\begin{proof}
We first show that if $a_i<b_j$ is a part of the weight interlacing condition, then $(\epsilon_n)_i=(\epsilon_{n-1})_j$. We will check that this holds for the classical branching law in the next example, and it's easy to see that changing the order of two consecutive weights preserve the property, so this is proven. It follows that if we change $a_i<b_j$ to $b_j<a_i$, then only $(\epsilon_n)_i$ and $(\epsilon_{n-1})_j$ gets flipped, so $p_n-p_{n-1}$ does not change. Again by the next example, $p_n-p_{n-1}$ takes the prescribed value in the case of the classical branching law, so we are done.
\end{proof}

Finally, we give two examples to illustrate the above combinatorial recipe.

\begin{example}\label{ex:Compact}
This is compatible with the classical branching law for compact groups. Recall that if $\pi_n$ is the representation of the compact Lie group $\U(n)$ with highest weight $(w_1\geq w_2\geq\cdots\geq w_n)$, then $\pi_n|_{\U(n-1)}$ is a multiplicity-free direct sum of representations with highest weights $(v_1\geq v_2\geq\cdots\geq v_{n-1})$ satisfying the perfect interlacing condition
\[
  w_1\geq v_1\geq w_2\geq v_2\cdots\geq v_{n-1}\geq w_n
\]
The infinitesimal character $\underline{a}$ is equal to $\underline{w}+\rho_n$, where $\rho_n=\big(\frac{n-1}{2},\frac{n-3}{2},\cdots,\frac{1-n}{2}\big)$ is half of the sum of positive roots. Similarly for $\underline{b}$. Therefore, the interlacing condition in terms of infinitesimal characters is
\[
  a_1>b_1>a_2>\cdots>a_{n-1}>b_{n-1}>a_n
\]
The distinguished characters have binary words
\[
  \chi_n:\cdots 1010,\quad \chi_{n-1}:\cdots 0101
\]
If $n$ is even, then the Harish-Chandra codes for both groups are entirely 1, corresponding to the inner forms $\U(0,n-1)\subseteq\U(0,n)$. If $n$ is odd, then they are all 0, corresponding to the inner forms $\U(n-1,0)\subseteq\U(n,0)$. Note that in both cases, the pairs are relevant.
\end{example}

\begin{example}
This is also compatible with Blattner's formula in the case $\U(n,0)\subseteq\U(n,1)$ \cite{HechtSchmid75}. We consider the case $n$ is odd for simplicity, otherwise the relevant signature is $\U(0,n)\subseteq\U(1,n)$, but the discussion is identical if we had chosen a different odd Hermitian space as our base space.

We use the notations in \emph{loc.\ cit.\ } to state the formula. Suppose the Harish-Chandra parameter is
\[
  \lambda=(a_1,\cdots,a_{k-1},a_{k+1},\cdots,a_{n+1},a_k)
\]
The positive roots associated with this system are
\[
  R_c^+=\{t_i-t_j\,|\,1\leq i<j\leq n\},\ R_n^+=\{t_1-t_{n+1},\cdots,t_{k-1}-t_{n+1},-t_k+t_{n+1},\cdots,-t_n+t_{n+1}\}
\]
Let $(c_1,\cdots,c_{n+1})$ be a character whose components sum to 0, then $Q(c_1,\cdots,c_{n+1})$ is 1 if and only if $c_1,\cdots,c_{k-1}\geq 0$ and $c_k,\cdots,c_n\leq 0$. Since $(b_1>\cdots>b_n)$ is the infinitesimal character on $\U(n)$, we have $\mu+\rho_c=(b_1,\cdots,b_n,b_{n+1})$, where $b_{n+1}$ is chosen so that $\sum_{i=1}^{n+1}b_i=\sum_{i=1}^{n+1}a_i$ and correspond to the character on $\U(1)$. Moreover, we can compute that
\[
  \lambda+\rho_n=\Big(a_1+\frac{1}{2},\cdots,a_{k-1}+\frac{1}{2},a_{k+1}-\frac{1}{2},\cdots,a_{n+1}-\frac{1}{2},a_k+\frac{n-2k+2}{2}\Big)
\]
If the weight interlacing condition is
\[
  b_1>a_1>\cdots>b_{k-1}>a_{k-1}>a_k>a_{k+1}>b_k\cdots>a_{n+1}>b_n
\]
then it's easy to check that the multiplicity given by Blattner's formula is 1. Moreover, in a weight interlacing condition with $b_i>b_{i+1}$ and no $a_{i'}$ in between, the transposition $(i\ i+1)$ does not change the value of $Q$, but it changes the sign, so the total sum is 0. Therefore, we have described the exact weight interlacing condition distinguished by $\U(n,0)\subseteq\U(n,1)$.

The distinguished representation predicted by the local GGP conjecture satisfy $\chi_n=01\cdots 010$ and $\chi_{n+1}$ differs from $10\cdots 1010$ at the $k$-th bit. It is easy to see that they correspond exactly to the weight interlacing condition computed above.
\end{example}

\subsubsection{Global conjecture}
Going back to the RACSDC global representation $\Pi=\Pi_{n-1}\otimes\Pi_n$. Fix a global additive character $\psi:\A_\CM/\A\to\C^\times$. Using it, we can define the global root number
\[
  \varepsilon(\Pi)=\prod_v\varepsilon(\phi_{n-1,v}\otimes\phi_{n,v},\psi_v)\in\{\pm 1\}
\]
It is independent of the choice of $\psi$, and it is the sign of the functional equation of $L(s,\Pi)$ at the centre $s=\frac{1}{2}$. We have two cases.
\begin{description}
  \item[Coherent case] $\varepsilon(\Pi)=1$. There is a pair of Hermitian spaces $V_{n-1}\subseteq V_{n}$ and a cuspidal automorphic representation $\pi$ on $\U(V_{n-1})\times\U(V_n)$ such that
  \begin{enumerate}
    \item The base change of $\pi$ to $\CM$ is $\Pi$.
    \item For each non-split place $v$, the local component $\pi_v$ is the distinguished element of the Vogan packet specified by the local conjecture.
    \item For any $\varphi\in\pi$, there is a formula of the form
    \[
      \left|\int_{[\U(V_{n-1})]}\varphi(h)\,dh\right|^2=(\ast)\cdot L\Big(\frac{1}{2},\Pi\Big)\cdot L(1,\Pi,\mathrm{ad})^{-1}
    \]
    where the term $(\ast)$ consists of an elementary term and an explicit product of local terms depending on $\varphi$. It is non-zero for some choice of $\varphi$.
  \end{enumerate}
  
  By the discussion in \cite[Section 25]{GGPConjecture}, items (1) and (2) follows from the endoscopic classification, which is known by \cite{KMSW} since our $L$-parameters are generic. Item (3) is the Gan--Gross--Prasad formula, fully proven in our case by \cite{GGPStable} (since we are starting with a cusp form $\Pi$, our parameters are stable).
  
  \item[Incoherent case] $\varepsilon(\Pi)=-1$. In contrast to the previous case,
  \begin{enumerate}
    \item The collection of local Hermitian spaces distinguished by the local conjecture no longer fits together into a Hermitian space defined over $\CM/\sh{F}$.
    \item $L\big(\frac{1}{2},\Pi\big)=0$ trivially by the functional equation.
  \end{enumerate}
  
  By the Beilinson--Bloch--Kato conjecture, this vanishing of the central value should be explained by the presence of certain algebraic cycles. In the ``minimal weight'' case, where the infinitesimal character of $\Pi$ at each archimedean place is given by the pair
  \[
    \Pi_{n-1}:\big(-\frac{n-2}{2},-\frac{n-4}{2},\cdots,\frac{n-2}{2}\big),\quad \Pi_n:\big(-\frac{n-1}{2},-\frac{n-3}{2},\cdots,\frac{n-1}{2}\big),
  \]
  such an algebraic cycle is conjecturally constructed in \cite[\S 27]{GGPConjecture} and \cite{ZhangAFL1}. In the more general case where the archimedean weights are perfectly interlacing (as in Example~\ref{ex:Compact}), one may construct the realizations of such cycles in various cohomology theories using coefficients.
\end{description}

As will be explained later in this article, such a dichotomy is also present in Iwasawa theory, producing what we will call coherent and incoherent Iwasawa main conjectures. The incoherent case appears to be considerably deeper, and it actually implies a version of the coherent main conjecture.

\subsection{Galois representations}\label{sec:GalRep}

\subsubsection{General properties}
Let $\Pi$ be an RACSDC automorphic representation of $\GL_n(\A_\CM)$. By the work of many people (cf.~\cite{ChenevierHarris,Caraiani2012}), there is a geometric Galois representation
\[
  \rho_\Pi:\Gal_\CM\to\GL_n(\bar{\Q}_p)
\]
attached to $\Pi$ in the sense that
\begin{equation}\label{eq:L-function}
  L(s,\Pi)=L\Big(\iota^{-1}\rho_\Pi,s+\frac{n-1}{2}\Big).
\end{equation}
Recall that $\iota:\C\simeq\C_p$ was a fixed isomorphism. Since $\Pi$ is conjugate self-dual of sign $(-1)^{n-1}$, the same is true for $\rho_\Pi$ up to a twist, namely we have an isomorphism
\[
  \rho_{\Pi}^{\mathtt{c}}\simeq\rho_{\Pi}^\vee(1-n).
\]
Moreover, $\rho_{\Pi}$ is pure of weight $n-1$.

We recall the Hodge--Tate weights of $\rho_{\Pi}$ at places above $p$. This requires some bookkeeping. Let $\sigma\in\Sigma_\infty^+$ be an embedding $\CM\hookrightarrow\C$. The associated archimedean places of $\CM$ and $\sh{F}$ will also be denoted by $\sigma$. Let $(a_i^\sigma)_{1\leq i\leq n}\in(\Z+\frac{n-1}{2})^n$ be the infinitesimal character of $\Pi_\sigma$. The composition $\sigma_p:=\iota\circ\sigma:\CM\hookrightarrow\C_p$ induces a $p$-adic place $w$ of $\CM$. Let $v$ be the place of $\sh{F}$ below $w$, then $v=w\bar{w}$ since $p$ splits completely in $\CM$. The Hodge--Tate weights of $\rho_\Pi$ are
\[
  \Big(\frac{n-1}{2}-a_i^\sigma\Big)_{1\leq i\leq n}\text{ at }w,\quad \Big(\frac{n-1}{2}+a_i^\sigma\Big)_{1\leq i\leq n}\text{ at }\bar{w}.
\]
Recall that by our convention, the cyclotomic character has Hodge--Tate weight $-1$.

\subsubsection{Ordinary representations}
Let $\sigma,\sigma_p,v,w$ be as in the previous subsection. Suppose $\Pi$ is unramified at $w$ (equivalently $\bar{w}$), then $\Pi_w$ is an unramified representation of $\GL_n(\sh{F}_v)$. Write its Satake parameters in the form
\[
  p^{-a_1^\sigma}\iota^{-1}(\alpha_1^\sigma),\cdots p^{-a_n^\sigma}\iota^{-1}(\alpha_n^\sigma)\in\C
\]
so we have a list of $p$-adic numbers $\alpha_1^\sigma,\cdots,\alpha_n^\sigma\in\C_p$. The representation $\rho_\Pi$ is crystalline at $w$ \cite{CH18}, and its Frobenius eigenvalues are
\[
  \Big(p^{\frac{n-1}{2}-a_i^\sigma}\alpha_i^\sigma\Big)_{1\leq i\leq n},
\]
which can be read off from equation~\eqref{eq:L-function}.

We say $\Pi$ is \emph{ordinary} (at $p$) if
\begin{equation}\tag{ord}
  \begin{cases}
    \text{(i) $\Pi$ is unramified at all places above }p\\
    \text{(ii) } |\alpha_i^\sigma|_v=1\text{ for all $i$ and all $\sigma$}
  \end{cases}
\end{equation}
Equivalently, the Hodge polygon and the Newton polygon coincide for $D_\cris(\rho_\pi)$. It follows that the local Galois representation $\rho_\Pi|_{\Gal_{\CM_w}}$ is upper triangular for each $p$-adic place $w$ (cf.~\cite[Lemma 7.2]{SkinnerUrban}). More precisely, if $w\in\Sigma_p^+$, then there exists a $\Gal_{\CM_w}$-stable increasing filtration
\[
  0=\Fil_0\rho_\Pi\subseteq\cdots\subseteq\Fil_n\rho_\Pi=\rho_\Pi
\]
such that for $i=1,\cdots,n$, the graded pieces are
\[
  (\Fil_i/\Fil_{i-1})\rho_\Pi\simeq\alpha_i^\sigma\chi_\mathrm{cyc}^{a_i^\sigma-\frac{n-1}{2}},
\]
where we also use $\alpha_i$ to denote the unramified character such that $\alpha_i(\Fr_\p)=\alpha_i$. 

Choose a basis $\{v_1,\cdots,v_n\}$ for $\rho_\pi$ so that $\Fil_i\rho_\pi=\bra v_1,\cdots,v_i\ket$. Let $\{v_1^\vee,\cdots,v_n^\vee\}$ be the dual basis on $\rho_\Pi^\vee$. Using the identification $\rho_\Pi^\vee\simeq\rho_\Pi^\mathtt{c}(n-1)$, we view $\rho_\Pi^\vee$ as supported on the same space as $\rho_\Pi$. This induces a \emph{decreasing} $\Gal_{\CM_{\bar{w}}}$-stable filtration
\[
  \rho_\Pi=\Fil^0\rho_\Pi\supseteq\cdots\supseteq\Fil^n\rho_\Pi=0
\]
with $\Fil^i\rho_\Pi=\bra v_{i+1}^\vee,\cdots,v_n^\vee\ket$. The graded pieces hare are
\[
  (\Fil^{i-1}/\Fil^i)\rho_\Pi\simeq\alpha_i^{-1}\chi_\mathrm{cyc}^{-a_i+\frac{n-1}{2}}
\]
as $\Gal_{\CM_{\bar{w}}}$-representations.

\section{Framework and further speculations}\label{sec:Guess}
Let $\Pi=\Pi_{n-1}^\vee\otimes\Pi_n$ be an RACSDC representation of the group $\GL_{n-1}(\A_\CM)\times\GL_n(\A_\CM)$ which is ordinary at $p$. The contragredient is inserted, since we are taking the point of view of a restriction problem.  The various numerology will also be simpler.

In this section, we propose a framework to understand Iwasawa theory for $\Pi$. Recall that for each $\sigma\in\Sigma_\infty^+$, we also denote its induced places on $\sh{F}$ and $\CM$ by $\sigma$. In addition, let
\[
  (a_1^\sigma>\cdots>a_n^\sigma),\quad (b_1^\sigma>\cdots>b_{n-1}^\sigma)
\]
be the infinitesimal characters of $\Pi_{n,\sigma}$ and $\Pi_{n-1,\sigma}$ respectively. Their relative position plays a key role in our discussion.

\subsection{Interlacing and Selmer conditions}
By applying the constructions in \S\ref{sec:GalRep} to $\Pi_{n-1}$ and $\Pi_n$, we can define an $n(n-1)$-dimensional geometric Galois representation
\[
  \rho_\Pi:=\rho_{\Pi_{n-1}^\vee}\otimes\rho_{\Pi_n}(n):\Gal_\CM\to\GL_{n(n-1)}(\bar{\Q}_p)
\]
which is attached to $\Pi$ in the sense that
\[
  L\Big(s+\frac{1}{2},\Pi\Big)=L(\rho_\Pi,s)
\]
It is arithmetically conjugate-symplectic. As a result, $L(\rho_\Pi,0)$ is always a critical value. More precisely, the $L$-factor at the archimedean place $\sigma$ is
\[
  L_\sigma(\rho_\Pi,s)=\prod_{i=1}^n\prod_{j=1}^{n-1}\Gamma_\C\bigg(s+\abs{b_j^\sigma-a_i^\sigma}+\frac{1}{2}\bigg).
\]
It follows that the critical values are those $s$ such that $\abs{s}\leq\abs{b_j^\sigma-a_i^\sigma}+\frac{1}{2}$ for all $i,j,\sigma$.

The Bloch--Kato conjecture for $\Pi$ states that
\[
  \ord_{s=\frac{1}{2}}L(s,\Pi)\stackrel{?}{=}\dim_\CM \h^1_f(\CM,\rho_\Pi)
\]
Here, the subscript $f$ is the Bloch--Kato local condition. It is the unramified condition away from $p$, but the conditions at the two primes above $p$ are more subtle. However, it is easy to describe if $\rho_\Pi$ satisfies the Panchishkin condition
\begin{equation}
  \parbox{0.8\linewidth}{
  At each $p$-adic place $w$ of $\CM$, the representation $\rho_\Pi$ contains a $\Gal_{\CM_w}$-stable subspace $\rho_\Pi^-$ such that the Hodge--Tate weights of $\rho_\Pi^-$ are non-positive and the Hodge--Tate weights of $\rho_\Pi/\rho_\Pi^-$ are positive.}\tag{P}
\end{equation}
Since $\rho_\Pi$ is pure of weight $-1$, we know that
\[
  \h^1_f(\CM_\p,\rho_\Pi)=\ker(\h^1(\CM_\p,\rho_\Pi)\to \h^1(\CM_\p,\rho_\Pi/\rho_\Pi^-))
\]
by \cite[Proposition 3.3.2(2)]{NekovarParityIII}.

The Hodge--Tate weights of $\rho_\Pi$ are directly related to the relative sizes of the infinitesimal characters for $\Pi_n$ and $\Pi_{n+1}$. We describe it using a ``weight interlacing string''.

\begin{definition}\label{def:String}
  An \emph{interlacing string} is any string of length $2n-1$ with exactly $n$ $A$s and $n-1$ $B$s. 

  Given two regular algebraic representations $\Pi_{n-1,\infty}$ and $\Pi_{n,\infty}$ with infinitesimal characters $(b_1>\cdots>b_{n-1})$ and $(a_1>\cdots>a_{n})$ respectively. Define their weight interlacing string as follows: first write the numbers $a_1,\cdots,a_{n},b_1,\cdots,b_{n-1}$ in descending orders, then forget the subscript to obtain a string of the desired form.
  
  The weight interlacing string for $\Pi$ is the collection of the weight interlacing strings as above, indexed by the archimedean places of $\sh{F}$.
\end{definition}

Recall that $a_i\in\Z+\frac{n+1}{2}$, but $b_j\in\Z+\frac{n}{2}$, so they cannot be equal. Now if $\Pi$ is ordinary (and hence unramified) at $p$, then all Selmer conditions are Panchishkin, and they are naturally indexed by weight interlacing strings. This is very explicit. Fix a $p$-adic place $v$ and let $w$ be the place above $v$ in the $p$-adic CM type. As described in Section~\ref{sec:GalRep}, choose a basis $\{w_1,\cdots,w_{n-1}\}$ for $\rho_{\Pi_{n-1}}$ which is compatible with the filtration. Similarly, choose a basis $\{v_1,\cdots,v_{n}\}$ for $\rho_{\Pi_n}$. Then a basis for $\rho_\Pi$ is $\{v_i\otimes w_j^\vee\}$. Given an interlacing string $\square$, we can define a subspace
\[
  \rho_\square:=\bra v_i\otimes w_j^\vee\,|\,a_i-b_j>0\ket\subseteq\rho_\Pi.
\]
By its compatibility with filtration, this is $\Gal_{\CM_w}$-stable and therefore defines a Selmer condition. At the place $\bar{w}$, the dual space is
\[
  \rho^\mathtt{c}_\square:=\bra v_i^\vee\otimes w_j\,|\,a_i-b_j<0\ket\subseteq\rho^\mathtt{c}_\Pi.
\]
which is a $\Gal_{\CM_{\bar{w}}}$-invariant subspace.

Given a collection of interlacing strings $\square=(\square_\sigma)_{\sigma\in\Sigma_\infty^+}$, we may use it to define Selmer conditions at all $p$-adic places of $\sh{K}$ as above. Together with the unramified condition away from $p$, this gives a Selmer group $\h^1_\square(\CM,\rho_\Pi)$. If $\square$ is the weight interlacing string of $\Pi$, then the previous discussion shows that
\[
  \h^1_\square(\CM,\rho_\Pi)=\h^1_f(\CM,\rho_\Pi),
\]
so $\square$ defines the correct Selmer condition.

Pictorially, we view  one such subspace as a tableau, for example
\begin{center}
  \begin{tikzpicture}[scale=0.4]
    \draw [fill=gray] (0,0) rectangle (3,1);
    \draw [fill=gray] (0,0) rectangle (2,3);
    \node at (-1,0.4) {$a_1$};
    \node at (-1,1.4) {$a_2$};
    \node at (-1,2.4) {$a_3$};
    \node at (-1,3.4) {$a_4$};
    \node at (0.5,-1) {$b_3$};
    \node at (1.5,-1) {$b_2$};
    \node at (2.5,-1) {$b_1$};
    \draw (0,0) grid (3,4);
  \end{tikzpicture}
\end{center}
In this diagram, each square represents a basis element, and the filled in area represents $\rho_\square$, consisting of the subset of basis elements defined by the inequality $a_i>b_j$. Since the sequences $\underline{a}$ and $\underline{b}$ are decreasing, the shaded part is lower-left closed. The same combinatorial structure appears in \cite{GHL} to determine the transcendental period.

\begin{example}
  Let $\sh{F}=\Q$ for simplicity. The case $n=2$ is the case of a modular form $f$ twisted by an anticyclotomic Hecke character $\chi$. In this case, there are 3 weight interlacing strings: $ABA,\ AAB,\ BAA$. The filtration on $\rho_f=\rho_{\Pi_2}$ is the usual ordinary filtration on the Tate module. Therefore, in more classical notations for Selmer groups, we have
  \[
    \h^1_{ABA}=\h^1_\mathrm{ord},\quad \h^1_{AAB}=\h^1_{\mathrm{rel},\mathrm{str}},\quad \h^1_{BAA}=\h^1_{\mathrm{str},\mathrm{rel}}
  \]
  At the archimedean place, if $f$ has weight $k$, then its infinitesimal character is $\big(\frac{k-1}{2},-\frac{k-1}{2}\big)$. There is a subtlety in identifying modular forms with automorphic forms on $\U(1,1)$, but this is fine if $f$ has trivial nebentype, which implies $k\in 2\Z$. The infinitesimal character of $\chi$ is $(\ell)$, with $\ell\in\Z$. If $\abs{\ell}<\frac{k-1}{2}$, then we are in the first setting $ABA$. If $\abs{\ell}>\frac{k-1}{2}$, then we are in one of the last two cases depending on the sign of $\ell$.
\end{example}

\subsection{Hypothesis on Hida family}\label{ss:BigGaloisRep}
This section recalls some aspects of Hida theory on a unitary group. Let $\mathtt{V}$ be an $n$-dimensional Hermitian space with respect to $\CM/\sh{F}$, and let $\mb{G}=\U(\mathtt{V})$. Since $p$ splits completely in $\CM$, it contains a maximal torus of the form
\[
  \mathbf{T}_n=\Res_{\sh{F}/\Q}\G_m^n.
\]
Following the notations introduced in \S\ref{ss:WeightSpace}, we have the weight space $\underline{\sh{W}}$. Given a classical point $x\in\underline{\sh{W}}_n$, let $(w_1^\sigma,\cdots,w_n^\sigma)_{\sigma\in\Sigma_p^+}$ be its collection of weights. Its \emph{unitary weight} is the tuple $(a_1^\sigma,\cdots,a_n^\sigma)_{\Sigma_\infty^+}$ defined by
\[
  a_i^{\sigma_\infty}=w_i^{\iota\circ\sigma_\infty}+\frac{n-1}{2}\in\Z+\frac{n-1}{2}
\]
This definition reflects the shifts used in defining Galois representations in \S\ref{sec:GalRep}.

Let $\underline{\Lambda}_n=\Or[[\mb{T}_n(\Z_p)]]$. For $i=1,\cdots,n$, there are standard cocharacters $\epsilon_i:\Res_{\sh{F}/\Q}\G_m\to\mathbf{T}_n$. Form the tautological character $\bs{\chi}_i:\Gal_{\CM}\to\underline{\Lambda}^\times$ as in equation~\eqref{eqn:TautChar}. For a fixed tame level, let $\underline{\mathbb{I}}_n$ be the big ordinary Hecke algebra for $\mb{G}$, cf.~for example \cite[\S 7.1]{EHLS}, where it is denoted by $\mathbf{T}^{\mathrm{ord}}_{K_r,R}$ for $R=\Or$. By Hida's control theorem (Theorem 7.2.1 of \emph{op.~cit.}), $\underline{\mathbb{I}}_n$ is finite locally free over $\underline{\Lambda}_n$.

\begin{hyp}
  There exists a $p$-adic family of Galois representations $\bs{T}_n$ of rank $n$ over $\underline{\mathbb{I}}_n$ which is conjugate self-dual in the sense of Definition~\ref{def:CSD}. Moreover, for each $w\in\Sigma_p^+$, it comes with an increasing $\Gal_{\CM_w}$-stable filtration
  \[
    0=\Fil_0\bs{T}_n\subseteq\cdots\subseteq\Fil_n\bs{T}_n=\bs{T}_n
  \]
  whose graded pieces are
  \[
    (\Fil_i/\Fil_{i-1})\bs{T}_n\simeq\bs{\alpha}_i^w\bs{\chi}_i|_{\Gal_{\CM_w}}
  \]
  Finally, there exists a point of $x_{\Pi_n}:\underline{\mathbb{I}}_n\to\bar{\Q}_p$ such that the specialization of $\bs{T}_n$ at $x_{\Pi_n}$ is the Galois representation attached to $\Pi_n$ as in \S\ref{sec:GalRep}.
\end{hyp}

In the residually multiplicity-free case, the existence of $\bs{T}$ together with the ordinary filtration follows from the construction of a $p$-refined pseudocharacter in \cite[\S 7.5]{BC09} together with their result lifting pseudocharacters to representations (Proposition 1.6.1 of \emph{op.~cit.}). However, it is not clear that conjugate self-duality, in the strong sense we require, follows from the construction, so we are stating it as a hypothesis.

\subsection{Iwasawa theory: algebraic side}\label{ss:IwAlg}
Apply the discussion of the previous subsection to $\Pi_{n-1}^\vee$ and $\Pi_n$. This gives a full weight space
\[
  \underline{\sh{W}}:=\underline{\sh{W}}_{n-1}\times\underline{\sh{W}}_n
\]
of dimension $(2n-1)g$. Given a collection of weight interlacing strings $\square=(\square_\sigma)_{\sigma\in\Sigma_\infty^+}$, we can define a subset of classical points
\[
  \underline{\sh{W}}^\square\subseteq\underline{\sh{W}}^\mathrm{cl}
\]
as those points whose \emph{unitary} weights interlace according to $\square$ at each $p$-adic place. This is Zariski dense for any choice of $\square$. As $\square$ runs over all tuples of interlacing strings, the sets $\underline{\sh{W}}^\square$ partition the set of classical points.

Over the weight space, we have a finite flat Hecke algebra $\underline{\mathbb{I}}$ and a $p$-adic family of Galois representations
\[
  \bs{T}:=\bs{T}_{n-1}\otimes\bs{T}_n(n)
\]
It is conjugate self-dual of rank $n(n+1)$ over $\underline{\mathbb{I}}$. Using the filtrations on $\bs{T}_{n-1}$ and $\bs{T}_n$, we can define $\Gal_{\CM_w}$-invariant subspaces
\[
  \bs{T}_\square\subseteq\bs{T}
\]
for each interlacing string $\square$, analogous to the definition of $\rho_\square\subseteq\rho_\Pi$. Given a collection $\square=(\square_\sigma)_{\sigma\in\Sigma_\infty^+}$, we obtain a Greenberg-type local condition for each pair of places $(w,\bar{w})$ of $\CM$ above $p$, as in equation~\eqref{eqn:GreenbergComplex}. All together, we can define the Selmer complex
\[
  \widetilde{\RG}_\square(\CM,\bs{T})
\]
and its cohomology groups $\widetilde{\h}^\ast_\square(\CM,\bs{T})$.

By construction, $\rho_\Pi$ is a specialization of $\bs{T}$. The next result follows easily from \cite[Proposition 2.2.4]{NekovarParityIII}.
\begin{lemma}\label{lem:Constant}
  There exists $\varepsilon_{\bs{T}}\in\{\pm 1\}$ such that for all classical points $x:\underline{\mathbb{I}}\to\bar{\Q}_p$, the specialization $\bs{T}_x$ has finite $\varepsilon$-factor equal to $\varepsilon_{\bs{T}}$.
\end{lemma}

Therefore, the global root number at a classical point $x$ is
\[
  \varepsilon(\bs{T}_x)=\varepsilon_{\bs{T}}\varepsilon_\infty(\bs{T}_x)
\]
By a computation similar to the one made in Section~\ref{sec:LocalGGP}, the archimedean root number (for the fixed additive character $\psi_\C$) counts the number of positive Hodge--Tate weights, so it is a function of the weight interlacing string. We denote it by $\varepsilon(\square)$. In particular, if $\varepsilon_{\bs{T}}\varepsilon(\square)=-1$, then the central $L$-values of all $x\in\underline{\sh{W}}^{\square}$ vanish for sign reasons. Taking this into account, we need both the coherent and incoherent versions of Iwasawa main conjecture.
\begin{conj}\label{conj:MainConjecture}
  Let $\square$ be a tuple of interlacing strings. There are two cases
  \begin{itemize}
    \item \textbf{Coherent case} If $\varepsilon_{\bs{T}}\varepsilon(\square)=1$, then $\widetilde{\h}^2_\square(\CM,\bs{T})$ is torsion, and
    \[
      \Char_{\mathbb{I}}\widetilde{\h}^2_\square(\CM,\bs{T})=(\sh{L}_p^\square)^2
    \]
    where $\sh{L}_p^\square$ is a $p$-adic measure whose square interpolates the central $L$-values of points in $\sh{W}^\square$, to be discussed in the following subsection.
    \item \textbf{Incoherent case} If $\varepsilon_{\bs{T}}\varepsilon(\square)=-1$, then $\widetilde{\h}^2_\square(\CM,\bs{T})$ has rank 1 over $\mathbb{I}$, and there is a special class
    \[
      \bs{z}\in\widetilde{\h}^1_\square(\CM,\bs{T})
    \]
    such that
    \[
      \Char_{\mathbb{I}}\widetilde{\h}^2_\square(\CM,\bs{T})_\mathrm{tors}=\Char_{\mathbb{I}}\left(\frac{\widetilde{\h}^1_\square(\CM,\bs{T})}{\mathbb{I}\cdot\bs{z}}\right)^2
    \]
  \end{itemize}
\end{conj}

Our main result Theorem~\ref{thm:Equiv} was axiomatized to relate these two conjectures. To relate them, we introduce the following notion.
\begin{definition}\label{def:Nearby}
  We say two interlacing strings $\square$ and $\triangle$ are \emph{nearby} if one if obtained from another by replacing a substring $AB$ with $BA$. Two weight interlacing strings are nearby if they are nearby at exactly one archimedean place, and they are the same elsewhere.
\end{definition}

On the tableau, this corresponds to the operation of deleting an extremal square, for example
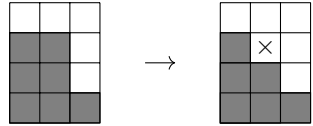
\begin{center}
  \begin{tikzpicture}[scale=0.4]
    \draw [fill=gray] (0,0) rectangle (3,1);
    \draw [fill=gray] (0,0) rectangle (2,3);
    \draw (0,0) grid (3,4);
    \draw [->] (4.5,2) -- (5.5,2);
    \begin{scope}[xshift=7cm]
	    \draw [fill=gray] (0,0) rectangle (3,1);
      \draw [fill=gray] (0,0) rectangle (1,3);
      \draw [fill=gray] (0,0) rectangle (2,2);
      \node at (1.5,2.5) {$\times$};
      \draw (0,0) grid (3,4);
    \end{scope}
  \end{tikzpicture}
\end{center}
The quotient $\bs{T}_\square/\bs{T}_\triangle$ is then the rank 1 space corresponding to the square deleted, so we can apply the theorem, at least over $\Lambda$. Further observe that $\varepsilon(\triangle)\varepsilon(\square)=-1$, so exactly one of them is incoherent. We can state the following expectation.
\begin{expect}
  Given a family $\bs{T}$, let $\triangle$ be an incoherent word for $\bs{T}$, i.e.\ $\varepsilon_{\bs{T}}\varepsilon(\triangle)=-1$. Suppose we have a ``special class'' $\bs{z}_\triangle\in\widetilde{\h}^1_\triangle(\CM,\bs{T})$, then
  \[
    \Char_{\mathbb{I}}\widetilde{\h}^2_\triangle(\CM,\bs{T})_\mathrm{tors} \supseteq \Char_{\mathbb{I}}\left(\frac{\widetilde{\h}^1_\triangle(\CM,\bs{T})}{\mathbb{I}\cdot\bs{z}_\triangle}\right)^2
  \]
  For any nearby word $\square$, let $\sh{L}_p^\square$ be the image of $\bs{z}_\triangle$ under the corresponding regulator map, then
  \[
    \Char_{\mathbb{I}}\widetilde{\h}^2_\square(\CM,\bs{T})\supseteq(\sh{L}_p^\square)^2
  \]
  Moreover, equality in any of the statements implies equality in all other statements.
\end{expect}

As explained in the remark following Theorem~\ref{thm:Equiv}, the first divisibility should follow from the construction of $\bs{z}_\triangle$ by an Euler system argument. However, we only have one example so far:
\[
  \triangle=\mathtt{diag},\quad\varepsilon_{\bs{T}}=-\varepsilon(\triangle)
\]
In this case, the method of \cite{LoefflerSpherical,LRZ-Moment} allows one to interpolate the usual diagonal cycles in a $p$-adic family, and the constructions of \cite{LaiSkinner} can be used to extend them to an Euler system. The above expectation relates this cycle to $p$-adic $L$-functions in the $2n$ tableaux nearby to it. For $n=2$, they are
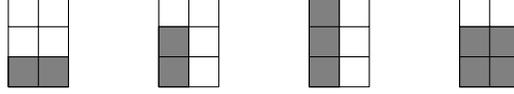
\begin{center}
  \begin{tikzpicture}[scale=0.4]
    \draw [fill=gray] (0,0) rectangle (2,1);
    \draw (0,0) grid (2,3);
    \begin{scope}[xshift=5cm]
	    \draw [fill=gray] (0,0) rectangle (1,2);
      \draw (0,0) grid (2,3);
    \end{scope}
    \begin{scope}[xshift=10cm]
	    \draw [fill=gray] (0,0) rectangle (1,3);
      \draw (0,0) grid (2,3);
    \end{scope}
    \begin{scope}[xshift=15cm]
	    \draw [fill=gray] (0,0) rectangle (2,2);
      \draw (0,0) grid (2,3);
    \end{scope}
  \end{tikzpicture}
\end{center}
Example~\ref{ex:DiagNearby} works out some numerology for constructing $p$-adic $L$-functions in this case.

Beyond the diagonal cycle, we are not aware of any systematic construction of Selmer classes. Moreover, if $\varepsilon_{\bs{T}}=\varepsilon(\mathtt{diag})$, then there is no systematic construction of Selmer classes for any weight interlacing string. It should be possible to construct them by degenerating $p$-adic families. For example, in the case $n=2$, the work of Castella--Tuan \cite{CastellaTuan} discussed in \S\ref{ss:NoHeegner} constructs the special class for the interlacing string $AAB$ using Gross--Kudla--Schoen cycles.

\subsection{Iwasawa theory: analytic side}\label{sec:Analytic}
From algebraic considerations, we see that different $p$-adic $L$-functions attached to each weight interlacing string. This can also be seen on the analytic side, though there are more mysterious phenomena here. We will explain some of them in the approach using coherent cohomology.

The global Gan--Gross--Prasad formula expresses the central $L$-value in terms of a period integral, which can potentially be interpolated to construct a $p$-adic $L$-function. The archimedean data, i.e. the signatures of $W,\ V$ and the nature of the discrete series at infinity, should determine the transcendental period. This is made precise in \cite{GHL}. Since the underlying geometry is so different, we are also not expecting a uniform $p$-adic $L$-function for all archimedean data. In fact, we propose the following imprecise conjecture.
\begin{conj}
  For each of the $\binom{2n-1}{n}^g$ weight interlacing strings $\square$, there exists a $p$-adic $L$-function $\sh{L}_p^\square$ in the Hecke algebra $\underline{\mathbb{I}}$ which interpolates $L\big(\frac{1}{2},\Pi\big)$ when the weights belong to $\underline{\sh{W}}^\square$.
\end{conj}

\begin{remark}
  This is an anticyclotomic $p$-adic $L$-function, in the sense that all of its points of interpolation corresponds to central $L$-values. It is different from the cyclotomic $p$-adic $L$-function recently constructed in \cite{DisegniZhang}. One could imagine extending the above conjecture to a $p$-adic $L$-function in $n(n+1)+1$ variables, incorporating cyclotomic deformations.
\end{remark}

One standard approach to construct $p$-adic $L$-functions is to interpret the period integral as a pairing in coherent cohomology of Shimura varieties. More precisely, suppose $\Pi$ is coherent, and $\mathtt{V}_{n-1}\subseteq \mathtt{V}_n$ is the relevant pair of Hermitian spaces distinguished by $\Pi$. For $N\in\{n-1,n\}$, we can form the unitary group $\mb{G}_N=\U(\mathtt{V}_N)$ and its associated Shimura variety $\Sh_N$. To construct the $p$-adic $L$-function, one would like to interpret the period integral
\[
  \int_{[\mb{G}_{n-1}]}\varphi_{n-1}^\vee(h)\varphi_n(h)\,dh
\]
as a pairing between coherent classes on $\Sh_{n-1}$. Let $(\square_\sigma)_{\sigma\in\Sigma_\infty^+}$ be the weight interlacing string for $\Pi$, then the local conjecture predicts a pair of Harish-Chandra codes $(\epsilon^{\mathrm{HC}}_{n-1,\sigma},\epsilon^{\mathrm{HC}}_{n,\sigma})$ for each archimedean place $\sigma$. Let $q_{n-1}$ and $q_n$ be the degrees described in \S\ref{ss:CoherentCoh} attached to this collection, then a basic numerology for the construction to succeed is
\begin{equation}\label{eqn:CohNumerology}
  q_{n-1}=q_n.
\end{equation}
For example, \cite{HarrisSqRoot} considers the case where the pair $(\Pi_{n-1},\Pi_n)$ distinguishes the holomorphic discrete series on both groups, so $q_{n-1}=q_n=0$. Note that we are using $\Pi=\Pi_{n-1}^\vee\otimes\Pi_n$, which is consistent with the use of \emph{anti}-holomorphic forms on the small group in \emph{op.~cit.}

We now compute what this means in terms of weight interlacing. In the case of Harris' work, the Harish-Chandra codes are
\[
  \epsilon_{n-1}^{\mathrm{HC}}=\underbrace{1\cdots1}_{q-1}\underbrace{0\cdots0}_p,\quad \epsilon_{n}^{\mathrm{HC}}=\underbrace{1\cdots1}_{q}\underbrace{0\cdots0}_p
\]
Using the recipe described in \S\ref{sec:LocalGGP}, we see that at every archimeden place, the weight interlacing tableau consists of removing a block of $p$ contiguous squares from $\mathtt{diag}$ from the end.
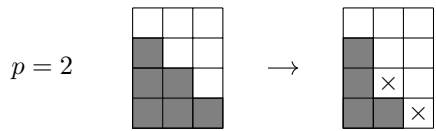
\begin{center}
  \begin{tikzpicture}[scale=0.4]
    \begin{scope}[xshift=-7cm]
      \node at (-3,2) {$p=2$};
	    \draw [fill=gray] (0,0) rectangle (3,1);
      \draw [fill=gray] (0,0) rectangle (1,3);
      \draw [fill=gray] (0,0) rectangle (2,2);
      \draw (0,0) grid (3,4);
      \draw [->] (4.5,2) -- (5.5,2);
    \end{scope}
    \draw [fill=gray] (0,0) rectangle (2,1);
    \draw [fill=gray] (0,0) rectangle (1,3);
    \node at (2.5,0.5) {$\times$};
    \node at (1.5,1.5) {$\times$};
    \draw (0,0) grid (3,4);
  \end{tikzpicture}
\end{center}
Technically, the above pair is only relevant if $n$ is even, and the relevant pair for odd $n$ are the anti-holomorphic representations. This is because of our choice of base point for the endoscopic labelling, cf.~equation \eqref{eqn:RelevantPair}. This does not affect our discussion.

\begin{remark}
  The construction in \cite{HarrisSqRoot} only gives a 1-variable $p$-adic $L$-function. In our framework, this weight deformation corresponds to the top-left square of the contiguous block. In the case represented by the previous picture, we would be considering the relations between the incoherent conjecture for $\triangle$ and the coherent conjecture for $\square$.
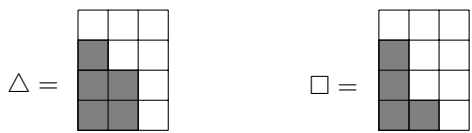
  \begin{center}
  \begin{tikzpicture}[scale=0.4]
    \begin{scope}[xshift=-10cm]
      \node at (-1.5,1.5) {$\triangle=$};
      \draw [fill=gray] (0,0) rectangle (1,3);
      \draw [fill=gray] (0,0) rectangle (2,2);
      \draw (0,0) grid (3,4);
    \end{scope}
    \node at (-1.5,1.5) {$\square=$};
    \draw [fill=gray] (0,0) rectangle (2,1);
    \draw [fill=gray] (0,0) rectangle (1,3);
    \draw (0,0) grid (3,4);
  \end{tikzpicture}
\end{center}
In the case $p=1$, $\triangle=\mathtt{diag}$, and $\square$ is nearby, so this 1-dimensional deformation is the exact one needed to relate $L$-values to the diagonal cycles.
\end{remark}

In general, removing an arbitrary set of squares from the diagonal results in a pair of Harish-Chandra codes such that $\epsilon_{n}^{\mathrm{HC}}$ is obtained from $\epsilon_{n-1}^{\mathrm{HC}}$ by appending a 1 to the front. The result still satisfies equation~\eqref{eqn:CohNumerology}, though the common degree is no longer 0, and higher Hida theory would be required to construct the relevant $p$-adic $L$-functions. By conjugation, the same holds if we add an arbitrary set of squares to the diagonal.

\begin{example}\label{ex:DiagNearby}
The $2n$ tableaux nearby to the diagonal cycles correspond to either adding or removing one square from the diagonal tableau. Since the diagonal tableau distinguishes the totally definite signature, these nearby tableaux distinguishes the signature $\U(n-2,1)\hookrightarrow\U(n-1,1)$. The graph of this embedding gives rise to the diagonal cycle.

On the other hand, the above computation shows that the $p$-adic $L$-function for these tableaux can potentially be constructed using cup products in coherent cohomology for the same embedding. Therefore, it is reasonable to expect that current methods can prove explicit reciprocity laws relating the diagonal cycle to all $2n$ nearby $p$-adic $L$-functions.
\end{example}

This does not account for all weight interlacing strings satisfying the numerology \eqref{eqn:CohNumerology}. The point is as follows: we have defined a function
\[
  \mathtt{dist}:\{\text{weight interlacing strings}\}\to\{\text{pairs of chambers}\}
\]
It turns out that both sides have the same size, but the map is not a bijection. Indeed, an operation of the form $AABB\to BBAA$ does not change the image. On a tableau, this corresponds to removing an extremal $2\times 2$ square, for example
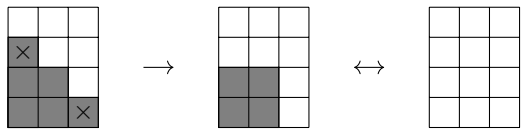
\begin{center}
  \begin{tikzpicture}[scale=0.4]
    \begin{scope}[xshift=-7cm]
	    \draw [fill=gray] (0,0) rectangle (3,1);
      \draw [fill=gray] (0,0) rectangle (1,3);
      \draw [fill=gray] (0,0) rectangle (2,2);
      \node at (0.5,2.5) {$\times$};
      \node at (2.5,0.5) {$\times$};
      \draw (0,0) grid (3,4);
      \draw [->] (4.5,2) -- (5.5,2);
    \end{scope}
    \draw [fill=gray] (0,0) rectangle (2,2);
    \draw (0,0) grid (3,4);
    \draw [<->] (4.5,2) -- (5.5,2);
    \begin{scope}[xshift=7cm]
      \draw (0,0) grid (3,4);
    \end{scope}
  \end{tikzpicture}
\end{center}

We suspect that this operation alone explains the fibres of $\mathtt{dist}$. If two words have the same image, then the archimedean data describing their period integrals are identical. This is a genuinely new phenomenon beyond the $\GL_2$-case. If $\square$ and $\square'$ are in the same fibre, then we expect there should be a relation between $\sh{L}_p^\square$ and $\sh{L}_p^{\square'}$. However, note that the description of the transcendental period in \cite{GHL} is different for $\square$ and $\square'$. Similarly, the $p$-adic Euler factors should also differ.

On the algebraic side, in the process of deleting the $2\times 2$ square, we see the following 6 regions:
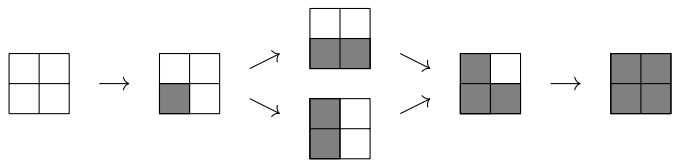
\begin{center}
  \begin{tikzpicture}[scale=0.4]
    \begin{scope}[xshift=-15cm]
      \draw (0,0) grid (2,2);
      \draw [->] (3,1) -- (4,1);
      \draw [fill=gray] (5,0) rectangle (6,1);
      \draw (5,0) grid (7,2);
    \end{scope}
    \draw [->] (-7,1.5) -- (-6,2);
    \draw [->] (-7,0.5) -- (-6,0);
    \begin{scope}[yshift=1.5cm]
      \draw [fill=gray] (-5,0) rectangle (-3,1);
      \draw (-5,0) grid (-3,2);
    \end{scope}
    \begin{scope}[yshift=-1.5cm]
      \draw [fill=gray] (-5,0) rectangle (-4,2);
      \draw (-5,0) grid (-3,2);
    \end{scope}
    \draw [->] (-2,2) -- (-1,1.5);
    \draw [->] (-2,0) -- (-1,0.5);
    \draw [fill=gray] (0,0) rectangle (2,1);
    \draw [fill=gray] (0,0) rectangle (1,2);
    \draw (0,0) grid (2,2);
    \draw [->] (3,1) -- (4,1);
    \draw [fill=gray] (5,0) rectangle (7,2);
    \draw (5,0) grid (7,2);
  \end{tikzpicture}
\end{center}
The two diagrams at the end are coherent. If we expect their $p$-adic $L$-functions to be related, then there should be a relation between the special classes for the diagrams in the second and fourth column. The simplest case is the relation between the diagonal cycle and the conjectural one attached to weight interlacing string $BABAA$. This new string distinguishes the archimedean data $\U(2,0)\hookrightarrow\U(2,1)$, with the form on $\U(2,1)$ in the holomorphic discrete series.

\bibliographystyle{alpha}
\bibliography{../../NT}

\end{document}

%% file: Preamble.tex
\usepackage{amsmath,amsthm,amssymb}
\usepackage[margin=2.5cm]{geometry}
\usepackage{enumitem}
\usepackage{tikz-cd}
\usepackage{mathrsfs}

\DeclareFontFamily{U}{wncy}{}
\DeclareFontShape{U}{wncy}{m}{n}{<->wncyr10}{}
\DeclareSymbolFont{mcy}{U}{wncy}{m}{n}
\DeclareMathSymbol{\Sha}{\mathord}{mcy}{"58} 

\def\bs{\boldsymbol}

\setlist[itemize]{label=--}

\newtheorem{theorem}{Theorem}[section]
\newtheorem{lemma}[theorem]{Lemma}
\newtheorem{prop}[theorem]{Proposition}
\newtheorem{cor}[theorem]{Corollary}
\newtheorem{conj}[theorem]{Conjecture}
\theoremstyle{definition}
\newtheorem{definition}[theorem]{Definition}
\newtheorem{example}[theorem]{Example}
\newtheorem{expect}[theorem]{Expectation}
\newtheorem{hyp}[theorem]{Hypothesis}

\newtheorem*{def*}{Definition}
\newtheorem*{ex*}{Example}
\theoremstyle{remark}
\newtheorem{remark}[theorem]{Remark}
\newtheorem*{remark*}{Remark}

\newcommand{\A}{\mathbb{A}}

\newcommand{\abs}[1]{\left|#1\right|}

\newcommand{\bra}{\langle}
\newcommand{\C}{\mathbb{C}}

\newcommand{\cont}{\mathrm{cont}}
\newcommand{\cris}{\mathrm{cris}}
\newcommand{\cyc}{\mathrm{cyc}}
\newcommand{\dR}{\mathrm{dR}}

\newcommand{\F}{\mathbb{F}}
\newcommand{\Fil}{\mathrm{Fil}}
\newcommand{\Fr}{\mathrm{Frob}}
\newcommand{\G}{\mathbf{G}}
\newcommand{\GL}{\mathrm{GL}}

\newcommand{\h}{\mathrm{H}}

\newcommand{\ket}{\rangle}

\newcommand{\mb}[1]{\mathbf{#1}}

\newcommand{\Norm}{\mathbf{N}}

\newcommand{\Or}{\mathcal{O}}
\newcommand{\p}{\mathfrak{p}}

\newcommand{\Q}{\mathbb{Q}}
\newcommand{\R}{\mathbb{R}}
\newcommand{\RG}{\mathbf{R}\Gamma}
\newcommand{\RHom}{\mathbf{R}\!\Hom}

\newcommand{\res}{\mathrm{res}}

\newcommand{\sh}[1]{\mathcal{#1}}

\newcommand{\tors}{\mathrm{tors}}
\newcommand{\U}{\mathrm{U}}

\newcommand{\Z}{\mathbb{Z}}

\renewcommand{\Im}{\operatorname{Im}}

\DeclareMathOperator{\Char}{char}
\DeclareMathOperator{\codim}{codim}
\DeclareMathOperator{\coker}{coker}

\DeclareMathOperator{\diag}{diag}

\DeclareMathOperator{\Ext}{Ext}
\DeclareMathOperator{\Frac}{Frac}
\DeclareMathOperator{\Gal}{Gal}
\DeclareMathOperator{\Hom}{Hom}

\DeclareMathOperator{\Log}{Log}

\DeclareMathOperator{\ord}{ord}

\DeclareMathOperator{\Res}{Res}

\DeclareMathOperator{\rank}{rank}

\DeclareMathOperator{\spf}{Spf}

%% file: GGP-IMC.bbl
\newcommand{\etalchar}[1]{$^{#1}$}
\begin{thebibliography}{KMSW14}

\bibitem[Ato20]{Atobe20}
Hiraku Atobe.
\newblock On the non-vanishing of theta liftings of tempered representations of
  {${\rm U}(p, q)$}.
\newblock {\em Advances in Mathematics}, 363:106984, 76, 2020.

\bibitem[BC09]{BC09}
Jo\"{e}l Bella\"{\i}che and Ga\"{e}tan Chenevier.
\newblock Families of {G}alois representations and {S}elmer groups.
\newblock {\em Ast\'{e}risque}, (324):xii+314, 2009.

\bibitem[BCK21]{BCK}
Ashay Burungale, Francesc Castella, and Chan-Ho Kim.
\newblock A proof of {P}errin-{R}iou's {H}eegner point main conjecture.
\newblock {\em Algebra Number Theory}, 15(7):1627--1653, 2021.

\bibitem[BD05]{BD-IMC}
M.~Bertolini and H.~Darmon.
\newblock Iwasawa's main conjecture for elliptic curves over anticyclotomic
  {$\mathbb{Z}_p$}-extensions.
\newblock {\em Ann. of Math. (2)}, 162(1):1--64, 2005.

\bibitem[BDP13]{BDP}
Massimo Bertolini, Henri Darmon, and Kartik Prasanna.
\newblock Generalized {H}eegner cycles and {$p$}-adic {R}ankin {$L$}-series.
\newblock {\em Duke Math. J.}, 162(6):1033--1148, 2013.
\newblock With an appendix by Brian Conrad.

\bibitem[BP15]{BPLocalRefined}
R.~Beuzart-Plessis.
\newblock Endoscopie et conjecture locale raffin\'{e}e de
  {G}an-{G}ross-{P}rasad pour les groupes unitaires.
\newblock {\em Compos. Math.}, 151(7):1309--1371, 2015.

\bibitem[BP16]{BPLocalWeak}
Rapha\"{e}l Beuzart-Plessis.
\newblock La conjecture locale de {G}ross-{P}rasad pour les repr\'{e}sentations
  temp\'{e}r\'{e}es des groupes unitaires.
\newblock {\em M\'{e}m. Soc. Math. Fr. (N.S.)}, (149):vii+191, 2016.

\bibitem[BPLZZ21]{GGPStable}
Rapha\"{e}l Beuzart-Plessis, Yifeng Liu, Wei Zhang, and Xinwen Zhu.
\newblock Isolation of cuspidal spectrum, with application to the
  {G}an-{G}ross-{P}rasad conjecture.
\newblock {\em Ann. of Math. (2)}, 194(2):519--584, 2021.

\bibitem[BSV22]{BSV1}
Massimo Bertolini, Marco~Adamo Seveso, and Rodolfo Venerucci.
\newblock Reciprocity laws for balanced diagonal classes.
\newblock To appear in Ast{\'e}risque, 2022.

\bibitem[Bur17]{BurungaleHeegII}
Ashay~A. Burungale.
\newblock On the non-triviality of the {$p$}-adic {A}bel-{J}acobi image of
  generalised {H}eegner cycles modulo {$p$}, {II}: {S}himura curves.
\newblock {\em J. Inst. Math. Jussieu}, 16(1):189--222, 2017.

\bibitem[Car12]{Caraiani2012}
Ana Caraiani.
\newblock Local-global compatibility and the action of monodromy on nearby
  cycles.
\newblock {\em Duke Math. J.}, 161(12):2311--2413, 2012.

\bibitem[CD23]{CastellaTuan}
Francesc Castella and Kim~Tuan Do.
\newblock Diagonal cycles and anticyclotomic {I}wasawa theory of modular forms.
\newblock Preprint, available at arXiv:2303.06751, 2023.

\bibitem[CH13]{ChenevierHarris}
Ga\"{e}tan Chenevier and Michael Harris.
\newblock Construction of automorphic {G}alois representations, {II}.
\newblock {\em Camb. J. Math.}, 1(1):53--73, 2013.

\bibitem[CH18]{CH18}
Francesc Castella and Ming-Lun Hsieh.
\newblock Heegner cycles and {$p$}-adic {$L$}-functions.
\newblock {\em Math. Ann.}, 370(1-2):567--628, 2018.

\bibitem[DZ24]{DisegniZhang}
Daniel Disegni and Wei Zhang.
\newblock {G}an--{G}ross--{P}rasad cycles and derivatives of {$p$}-adic
  {$L$}-functions.
\newblock Preprint, 2024.

\bibitem[EHLS20]{EHLS}
Ellen Eischen, Michael Harris, Jianshu Li, and Christopher Skinner.
\newblock {$p$}-adic {$L$}-functions for unitary groups.
\newblock {\em Forum of Mathematics. Pi}, 8:e9, 160, 2020.

\bibitem[GGP12]{GGPConjecture}
Wee~Teck Gan, Benedict~H. Gross, and Dipendra Prasad.
\newblock Symplectic local root numbers, central critical {$L$} values, and
  restriction problems in the representation theory of classical groups.
\newblock In {\em Sur les conjectures de Gross et Prasad. I}, number 346, pages
  1--109. 2012.

\bibitem[GHL21]{GHL}
Harald Grobner, Michael Harris, and Jie Lin.
\newblock {D}eligne's conjecture for automorphic motives over {CM}-fields.
\newblock Prepint, available at arXiv:1802.02958, 2021.

\bibitem[Gre94]{GreenbergMotives}
Ralph Greenberg.
\newblock Iwasawa theory and {$p$}-adic deformations of motives.
\newblock In {\em Motives ({S}eattle, {WA}, 1991)}, volume~55 of {\em Proc.
  Sympos. Pure Math.}, pages 193--223. Amer. Math. Soc., Providence, RI, 1994.

\bibitem[Har90]{Harris90}
Michael Harris.
\newblock Automorphic forms of {$\overline\partial$}-cohomology type as
  coherent cohomology classes.
\newblock {\em J. Differential Geom.}, 32(1):1--63, 1990.

\bibitem[Har21]{HarrisSqRoot}
Michael Harris.
\newblock Square root {$p$}-adic {$L$}-functions {I}: {C}onstruction of a
  one-variable measure.
\newblock {\em Tunis. J. Math.}, 3(4):657--688, 2021.

\bibitem[He17]{HeGGP}
Hongyu He.
\newblock On the {G}an-{G}ross-{P}rasad conjecture for {$U(p,q)$}.
\newblock {\em Invent. Math.}, 209(3):837--884, 2017.

\bibitem[How04a]{HowardDivisibility}
Benjamin Howard.
\newblock The {H}eegner point {K}olyvagin system.
\newblock {\em Compos. Math.}, 140(6):1439--1472, 2004.

\bibitem[How04b]{HowardGL2}
Benjamin Howard.
\newblock Iwasawa theory of {H}eegner points on abelian varieties of {$\rm
  GL_2$} type.
\newblock {\em Duke Math. J.}, 124(1):1--45, 2004.

\bibitem[HS75]{HechtSchmid75}
Henryk Hecht and Wilfried Schmid.
\newblock A proof of {B}lattner's conjecture.
\newblock {\em Invent. Math.}, 31(2):129--154, 1975.

\bibitem[HY23]{HsiehYamana23}
Ming-Lun Hsieh and Shunsuke Yamana.
\newblock Five-variable $p$-adic $l$-functions for $u(3)\times u(2)$.
\newblock Preprint, available at arXiv:2311.17661, 2023.

\bibitem[JNS24]{JNS}
Dimitar Jetchev, Jan Nekov\'{a}\v{r}, and Christopher Skinner.
\newblock Split {E}uler systems for conjugate-dual {G}alois representations.
\newblock Preprint, 2024.

\bibitem[JSW17]{JSW}
Dimitar Jetchev, Christopher Skinner, and Xin Wan.
\newblock The {B}irch and {S}winnerton-{D}yer formula for elliptic curves of
  analytic rank one.
\newblock {\em Cambridge Journal of Mathematics}, 5(3):369--434, 2017.

\bibitem[KLZ17]{KLZ1}
Guido Kings, David Loeffler, and Sarah~Livia Zerbes.
\newblock Rankin-{E}isenstein classes and explicit reciprocity laws.
\newblock {\em Camb. J. Math.}, 5(1):1--122, 2017.

\bibitem[KMSW14]{KMSW}
Tasho Kaletha, Alberto Minguez, Sug~Woo Shin, and Paul-James White.
\newblock Endoscopic classification of representations: inner forms of unitary
  groups.
\newblock Preprint, 2014.

\bibitem[Kol88]{Kolyvagin88}
V.~A. Kolyvagin.
\newblock Finiteness of {$E(\mathbf{Q})$} and {$\Sha(E,\mathbf{Q})$} for a
  subclass of {W}eil curves.
\newblock {\em Izv. Akad. Nauk SSSR Ser. Mat.}, 52(3):522--540, 670--671, 1988.

\bibitem[Liu23]{LiuRS}
Yifeng Liu.
\newblock Anticyclotomic {$p$}-adic {$L$}-functions for {R}ankin--{S}elberg
  product.
\newblock Preprint, available at \verb!https://arxiv.org/abs/2306.07039!, 2023.

\bibitem[Loe21]{LoefflerSpherical}
David Loeffler.
\newblock Spherical varieties and norm relations in {I}wasawa theory.
\newblock {\em Journal de Th{\'e}orie des Nombres de Bordeaux},
  33(3.2):1021--1043, 2021.

\bibitem[LRZ24]{LRZ-Moment}
David Loeffler, Robert Rockwood, and Sarah~Livia Zerbes.
\newblock Spehrical varieties and {$p$}-adic families of cohomology classes.
\newblock Preprint, available at arXiv:2106.16082, 2024.

\bibitem[LS24]{LaiSkinner}
Shilin Lai and Christopher Skinner.
\newblock Anti-cyclotomic {E}uler system of diagonal cycles.
\newblock Preprint, available at arXiv:2408.01219, 2024.

\bibitem[LTX{\etalchar{+}}22]{LTXZZ}
Yifeng Liu, Yichao Tian, Liang Xiao, Wei Zhang, and Xinwen Zhu.
\newblock On the {B}eilinson-{B}loch-{K}ato conjecture for {R}ankin-{S}elberg
  motives.
\newblock {\em Invent. Math.}, 228(1):107--375, 2022.

\bibitem[LTX24]{LTX}
Yifeng Liu, Yichao Tian, and Liang Xiao.
\newblock {I}wasawa's main conjecture for {R}ankin--{S}elberg motives in the
  anticyclotomic case.
\newblock Preprint, available at \verb!https://arxiv.org/abs/2406.00624!, 2024.

\bibitem[LZ21]{LZConjecture}
David Loeffler and Sarah~Livia Zerbes.
\newblock {$p$}-adic {$L$}-functions and diagonal cycles for
  {$\mathrm{GSp}(4)\times\mathrm{GL}(2)\times\mathrm{GL}(2)$}.
\newblock Preprint, 2021.

\bibitem[Nek06]{NekovarSC}
Jan Nekov\'{a}\v{r}.
\newblock Selmer complexes.
\newblock {\em Ast\'{e}risque}, (310):viii+559, 2006.

\bibitem[Nek07]{NekovarParityIII}
Jan Nekov\'{a}\v{r}.
\newblock On the parity of ranks of {S}elmer groups. {III}.
\newblock {\em Doc. Math.}, 12:243--274, 2007.

\bibitem[Och03]{OchiaiColemanMap}
Tadashi Ochiai.
\newblock A generalization of the {C}oleman map for {H}ida deformations.
\newblock {\em Amer. J. Math.}, 125(4):849--892, 2003.

\bibitem[PR87]{PRHeegnerPoint}
Bernadette Perrin-Riou.
\newblock Fonctions {$L$} {$p$}-adiques, th\'{e}orie d'{I}wasawa et points de
  {H}eegner.
\newblock {\em Bull. Soc. Math. France}, 115(4):399--456, 1987.

\bibitem[Ski18]{SkinnerAWS}
Christopher Skinner.
\newblock Lectures on the {I}wasawa theory of elliptic curves, 2018.
\newblock Arizona Winter School notes, available at
  \url{https://swc-math.github.io/aws/2018/index.html}.

\bibitem[Ski20]{SkinnerConverse}
Christopher Skinner.
\newblock A converse to a theorem of {G}ross, {Z}agier, and {K}olyvagin.
\newblock {\em Annals of Mathematics. Second Series}, 191(2):329--354, 2020.

\bibitem[SU14]{SkinnerUrban}
Christopher Skinner and Eric Urban.
\newblock The {I}wasawa main conjectures for {$\rm GL_2$}.
\newblock {\em Invent. Math.}, 195(1):1--277, 2014.

\bibitem[Tat79]{TateCorvallis}
J.~Tate.
\newblock Number theoretic background.
\newblock In {\em Automorphic forms, representations and {$L$}-functions
  ({P}roc. {S}ympos. {P}ure {M}ath., {O}regon {S}tate {U}niv., {C}orvallis,
  {O}re., 1977), {P}art 2}, Proc. Sympos. Pure Math., XXXIII, pages 3--26.
  Amer. Math. Soc., Providence, R.I., 1979.

\bibitem[Vat03]{Vatsal03}
V.~Vatsal.
\newblock Special values of anticyclotomic {$L$}-functions.
\newblock {\em Duke Math. J.}, 116(2):219--261, 2003.

\bibitem[Zha12]{ZhangAFL1}
Wei Zhang.
\newblock On arithmetic fundamental lemmas.
\newblock {\em Invent. Math.}, 188(1):197--252, 2012.

\end{thebibliography}
